\documentclass[a4paper]{article}
\usepackage{xcolor}
\usepackage{amsfonts}
\usepackage{amsmath}
\usepackage{amssymb}
\usepackage{graphicx}
\usepackage[dvipsnames]{xcolor}
\usepackage[breaklinks=true]{hyperref}
\usepackage{breakcites}
\usepackage{geometry}%
\providecommand{\U}[1]{\protect\rule{.1in}{.1in}}
\newtheorem{theorem}{Theorem}

\newtheorem{definition}[theorem]{Definition}

\newtheorem{lemma}[theorem]{Lemma}

\newtheorem{proposition}[theorem]{Proposition}
\newtheorem{remark}[theorem]{Remark}

\newenvironment{proof}[1][Proof]{\noindent\textbf{#1.} }{\ \rule{0.5em}{0.5em}}
\newcommand{\R}{\mathbb{R}}
\newcommand{\pa}{\partial}
\newcommand{\na}{\nabla}
\newcommand{\bfm}{\textbf{m}}
\newcommand{\bfM}{\textbf{M}}

\def\Xint#1{\mathchoice
{\XXint\displaystyle\textstyle{#1}}%
{\XXint\textstyle\scriptstyle{#1}}%
{\XXint\scriptstyle\scriptscriptstyle{#1}}%
{\XXint\scriptscriptstyle\scriptscriptstyle{#1}}%
\!\int}
\def\XXint#1#2#3{{\setbox0=\hbox{$#1{#2#3}{\int}$ }
\vcenter{\hbox{$#2#3$ }}\kern-.6\wd0}}

\def\dashint{\Xint-}

\numberwithin{equation}{subsection}
\numberwithin{theorem}{subsection}

\title{A variational approach to the derivation of\\ 
reduced models for bubbly flows} 
\author{Cosmin Burtea\thanks{cosmin.burtea@imj-prg.fr, Université Paris Cité, Sorbonne Université, CNRS, IMJ-PRG, F-75013 Paris, France}, David Gérard-Varet \thanks{david.gerard-varet@imj-prg.fr, Université Paris Cité, Sorbonne Université, CNRS, IMJ-PRG, F-75013 Paris, France}}

\begin{document}

\maketitle

\begin{abstract}
In this paper, we derive reduced models for the motion of gas bubbles in an ambient inviscid liquid, using Hamilton's least action principle. We first explain how to recover from this principle the classical sharp interface model, in which the pressure  is continuous across the surfaces of the bubbles. We then show how to reduce the complexity of the model, by simplifying the description of those surfaces. Namely, we impose them to evolve within a subclass of hypersurfaces described by a finite number of parameters (the simplest example being spheres, that is neglecting deviation of the bubbles from sphericity). The difficulty from a mathematical and modeling point of view is to determine the interface conditions that substitute to pressure continuity. We complete the derivation of the reduced models by some well-posedness analysis, in the case of curl-free liquid flow and homogeneous pressure in the bubbles.

\end{abstract}

\section{Introduction}

\subsection{Context}

\ \ \ \ When waves propagate in liquids, the strongest acoustic feedback is due to
the presence of gas bubbles \cite{leighton}. The sound speed in a bubbly flow
is greatfully impacted by the volume fraction of the bubbles and its minimal
value can be much lower than the sound speeds in either fluid medium. For example,
the Wood sound speed model \cite{Wood1930} predicts that for air-water
mixtures with $c_{1}=330m/s,\rho_{1}=1,24kg/m^{3},$ $c_{2}=1500m/s,$ $\rho
_{2}=1000kg/m^{3}$ the minimal sound speed can be as low as $23m/s$ . This
surprising fact has been validated experimentally \cite{Micaelli1982}.

\medskip

Applications of bubble acoustics range from medicine (ultrasonic scanning)
to military industry (sonars) or oceanography to name just a few. Important
information can be inferred from quantitative and qualitative (size
distribution) estimates of the bubble population : dilation/contraction of
bubbles leads to a modification of the total surfacic area and has impact on
mass and heat transfer \cite{Legendre_et_al2012}. This kind of information can
be used in oceanography for a better understanding of coastal erosion,
exchanges of mass, energy and temperature between the ocean and the
atmosphere, in the petrochemical industry in order to optimize harvest and
transport of petrol/gas, or in cosmetics and pharmaceutical industries where
the control of the bubbles introduced in the fabrication process is crucial
for the lifespan of the final products.

We refer to the remarkable review paper \cite{leighton} where all these
examples and many others are discussed in detail.

\medskip

Bubbly flows belong to the general family of two-phase flows. When seen as
continuous media, molecular differences between the two phases cause them to
deform differently under the action of a given force. Thus the two phases
undergo two different nonlinearly coupled dynamics \cite{risso2018agitation}.
Depending on the scale at which one wants to describe the interaction between
the two phases we distinguish two types of models : sharp interfaces or
diffuse interface models. 

\medskip

In this paper we propose several models of the former type for describing the
motion of gas bubbles in an ambient liquid : we consider the case where the
phases are completely separated by sharp interfaces but where the continuum
hypothesis still holds for each one of the two phases. Since each phase
occupies its own domain, in order to obtain a closed model, \textit{one must
encode the interactions of the phases at the level of the interfaces}.

\medskip

In full generality, the dynamics of the two phases are governed by two Euler or
Navier-Stokes equations (with various degrees of complexity viscous/inviscid,
incompressible/compressible etc.) for the bulk phases that are classically coupled at the
level of the interface  through the equality of normal velocities
and normal stresses. This is a free boundary problem which is heavy to handle
both numerically or analytically. For mathematical results on the well-posedness of this type of equations we refer the reader to \cite{Beale,Tani,Guo-Tice} regarding the Navier-Stokes case, and to the book \cite{Lannes} for the inviscid equations. The specific case of bubbles in a viscous incompressible flow was tackled in \cite{Marcel}. Let us point out that in the case of a single bubble, under radial symmetry of the initial data, the sphericity of the bubble is preserved, as well as the symmetry of the fields, providing a simpler set of equations. We refer to \cite{Biro} for an in-depth analysis of the well-posedness properties of such system. Recently, a nice  stability analysis of this radially symmetric setting was conducted in \cite{Weinstein}. See also \cite{Weinstein2} with an additional periodic forcing, and \cite{Weinstein3} where the introduction of asymmetric deformations of the bubbles is shown to trigger some strong instability. 

\medskip

Still, even in settings where radial symmetry is not strictly preserved, bubbles remain often close to spherical and/or ellipsoidal. This is observed in many applications and for a wide range of Reynolds and E\"{o}tv\"{o}s numbers \cite{ellingsen2001rise,van2005numerical}. In the same spirit, the recent article \cite{JangTice} shows that under the action of gravity, bubbles in hydrostatic equilibrium deviate slightly from sphericity, with a deformation along the vertical axis associated with a one dimensional bifurcation phenomenon. 
Taking these observations as a starting point, a natural question arises: can one derive \textit{reduced models}, in which the surfaces of the bubbles have to  remain spherical, or ellipsoidal, or more generally have to evolve within a subclass described by a finite number of parameters ?

Of course there is no reason to imagine that the bulk
equations should change. The difficulty from a mathematical
and modelling point of view is \textit{to determine the appropriate interface
conditions that govern the reduced dynamics.} The equality
of normal velocities can be understood geometrically as a necessary condition
for the liquid and the gas to maintain a common interface and therefore we
will also encounter it in the reduced models. However the equality of normal
stresses has to be relaxed owing to the imposed shape of the
interface. How to relax this condition and still keep a
well-posed system of PDEs is the focus of the present paper. 

\medskip

Regarding spherical bubbles, simplified models in the spirit of the famous
Rayleigh-Plesset equation were obtained in previous works
\cite{hermans1973},\cite{Smereka}, \newline \cite{Gavrilyuk}. We mention that in these papers the
irrotationality of the velocity field is exploited in order to drastically
simplify the dynamics : the latter is governed by a system of ODEs for the
radii and centers of the gas bubbles. We also refer to our recent contribution \cite{Nous}, where we build weak Leray solutions for a reduced system modeling the interaction between a viscous liquid and spherical bubbles. The conditions at the interfaces there are viscous analogues of some of the conditions that will be derived in this paper. Another reference is \cite{Hillairet1}, in which a PDE system governing the evolution of a finite number of spherical gas bubbles is obtained through a singular limit approach. Namely, the bubbles are modeled by the viscous compressible Navier-Stokes equations and the shear viscosity is sent to infinity. 

\medskip

Our goal here is to go  beyond the case of spherical bubbles and to obtain some new PDE models in the framework of the Hamilton's least action
principle : our models can be seen as the Euler-Lagrange
equations verified by critical points of certain action functionals. The
advantage of this method is that the whole physics is contained in the
definition of a scalar function called the Lagrangian of the system, which is
used to construct the action functional, \textit{and in the class where the
critical point is sought}. As it is the case with finite dimensional systems, invariance of the Lagrangians gives rise to conservation laws \cite{Gouin1976,Serre2018}. The variational approach to derive equations in
fluid mechanics and more general equations of continuous media has a long
history going back to Serrin's work \cite{Serrin1959} and continues to this
day \cite{Gay-BalmazYoshimura2017}. We refer the reader to the (by
far non-exhaustive) list of important contributions
\cite{Eckart1960,Arnold1966,salmon1988hamiltonian,Godunov1978,HolmRatiuMarsden1998,GavrilyukGouin1999,Berdichevsky2009}. For more elementary introductory papers oriented towards modeling, see
\cite{Gavrilyuk2011,BurteaGavrilyukPerrin2024}. The equations governing the
interaction between rigid bodies and a liquid were obtained in
\cite{HouotMunnier,Vankerschaver_et_al2009,GlassSueur2012}. This problems bears some
structural resemblances with the liquid-bubbles interaction, however in the former case the whole dynamics of the solids and of their boundaries is driven by rigid motions and interface conditions are a bit easier to understand. 

\medskip

Although  we do not study this kind of models in this paper, let us say a few words on the other class of models that we
briefly mentioned : diffuse interface models. In bubbly flows or more general
dispersed two phase-flows, the number of interfaces is or becomes too large
for a sharp-interface model to be used for numerical simulations. For
this reason, one attempts to derive models for the "average flow" that
contains macroscopic information about the mixture. In this description, a
single point in space can "accommodate" different fluid particles. New
variables, volume fractions, are introduced in order to quantify
the space occupied by each fluid particle. Macroscopic models are obtained in
the literature either by averaging/homogenization techniques from mesoscopic
models
\cite{drew2006theory,ishii2010thermo,PerrierGutierrez2021,Hillairet2,Bresch_et_al2024}, directly from Hamilton's stationary action principle
\cite{Saurel_et_al2003,Gavrilyuk,GavrilyukSaurel2002} or using methods of statistical physics developed for description of many particle systems \cite{Caflisch,JabinPerthame2000,Gavrilyuk}. Of course, this very
short list of references is incomplete and only reflects the
authors's biased mathematical and modelling culture.

\subsection{Outline of the paper}

Our paper is structured as follows. 

We start in Section \ref{sec_Lag_approach} with the derivation of the classical two-phase flow system: Euler's equations in the bulk phases, coupled through the equality of the normal velocities and of the pressures at the interface(s). We adopt a lagrangian point of view:   we describe the liquid and gas phases with one-parameter families of diffeomorphims $X_\pm=X_\pm(t,x)$ that describe fluid particles trajectories. We first collect in Section \ref{sec_Some_general_consid} some notations and  useful formulae which allow to switch from lagrangian to eulerian formulations and vice-versa.  Then, in Section \ref{section_immiscible}, we recover the Euler free boundary two-fluid system as the Euler-Lagrange equations verified by the critical point of an action functional, that is the difference between the kinetic and potential energies. The action functional is sought in the class of diffeomorphisms that maintain a common interface. To be a bit more precise, if $\mathcal{I}_0$ is the initial liquid-gas interface, the interface separating the two phases at time $t>0$ is given by
\begin{equation}
\label{interface_intro}
\mathcal{I}(t)=X_+(t,\mathcal{I}_0)=X_-(t,\mathcal{I}_0).
\end{equation} 
 Unlike in the bulk phases, perturbations $\delta X_\pm$ cannot be chosen arbitrarily at the level of the interfaces : they should also lead to configurations at time $t>0$ that maintain a common interface 
 \[
(X_++\delta X_+)(t,\mathcal{I}_0)=(X_-+\delta X_-)(t,\mathcal{I}_0).
\]
This is the key observation leading to the condition $p_+=p_-$ verified at the level of the interface.


\medskip
In Section \ref{sec_shape_constrained}, we turn to the core of our work, that is the derivation of simplified models with constrained deformations for the bubbles. The main idea is that those simpler models can be obtained by searching for critical points of the action functional in a narrower class of diffeomorphisms : besides \eqref{interface_intro} we ask  $\mathcal{I}(t)$ to be the union of a finite number of elements drawn from a collection $(\mathcal{S}_m)_{m\in M}$
where for each $m \in M$,  $\mathcal{S}_m$ is the boundary of a smooth bounded domain $\Omega_m$ of $\R^3$. One should think of this collection as the surfaces that result from admissible deformations of a fixed prototype: $M$ is a finite dimensional manifold representing the finite number of degrees of freedom allowed for the deformations of the bubbles.

Accordingly, the class of admissible perturbations $\delta X_\pm$ is much smaller: they need to preserve the constraint that each connected component of $\mathcal{I}(t)$ belongs to  $(\mathcal{S}_m)_{m\in M}$ at each time $t$.  this leads to relaxing the condition  $p_+=p_-$ at the interface. Roughly, this relation is replaced by a finite number of integral conditions  of the form $\int_\mathcal{I}(p_+-p_-)v_nd\sigma=0$, where the fields $v_n = v_n(x)$ form a basis of all possible normal velocity fields associated with deformations of the interfaces within the subclass $(\mathcal{S}_m)_{m\in M}$. All details are provided in Section \ref{sec_formalism}.

For instance, in the simplest case of spherical bubbles, each bubble is described by a four-dimensional set of parameters $m = (c,r)$ where $c = (c_1,c_2,c_3)$ is its center and its radius $r$. Deformations preserving sphericity are described by normal velocities of the form $\dot{c}n+\dot{r}$ where $n$ is the normal at the spherical bubble while $\dot{c},\dot{r}$ can be chosen freely. This leads to exactly four boundary conditions at the level of the boundary of each bubble : $\int_\mathcal{I}(p_+-p_-)nd\sigma=0$ respectively $\int_\mathcal{I}(p_+-p_-)d\sigma=0$. See Section \ref{sec_spherical_bubbles} for more on the special case of spherical bubbles, and Section \ref{sec_ellipsoidal_bubbles} for ellipsoidal bubbles. Two simplified models are further obtained in the case of spherical bubbles in Section \ref{sec_simplified_models}.

Those simple cases are there to illustrate our methodology, but do not exhaust its applications.  More sophisticated frameworks incorporating more complicated physics could be obtained. We speculate that this kind of reduced models could be notably useful for computer simulations. Indeed, a numerical description of bubbles interfaces can only be achieved approximately, that is with a finite number of degrees of freedom, so that reduced models might be natural and computationally less expensive.  

\medskip
Of course, the derivation carried in Section \ref{sec_shape_constrained} is a matter of mathematical modeling. It relies on the choice of an action functional and is about the Euler-Lagrange equations that a critical point of this functional (if there is any) should satisfy. The actual existence of such critical points or of solutions to the reduced models is a separated mathematical problem. In the last Section \ref{sec_WP_reduced}  of the present study we briefly touch upon this question for the case of irrotational velocity fields for the liquid and space homogeneous pressures for the gaseous phase. We show that this leads to a well-posed set of ODEs that generalizes the one obtained in \cite{Smereka,Gavrilyuk} for the spherical case.


\medskip

\section{Lagrangian approach for immiscible fluids}\label{sec_Lag_approach}
In this section, we explain how to derive formally the classical equations for immiscible flows. We focus on the case where one phase is compressible (gaseous bubbles) and one incompressible (the liquid in which the bubbles are immersed).  The point of view that we adopt here is lagrangian: we derive Euler-Lagrange equations from a least action principle, and then convert them into the usual eulerian formulation. In the next section, this principle will be modified to account for additional shape constraints on the interfaces.

\subsection{Some general considerations}\label{sec_Some_general_consid}
We first introduce several general notions, related to the lagrangian description of continuous deformable media.  

\begin{definition}
A continuous deformable medium of the initial configuration $\Omega_{0}$ is a
pair $\left(  \Omega_{0},X\right)  $ such that

\begin{itemize}
\item $\Omega_{0}\subset\mathbb{R}^{d}$ is a given open set which represents
the initial configuration of the medium;

\item $X:[0,\infty)\times\overline{\Omega_{0}}\rightarrow\mathbb{R}^{d}$ is a
continuously differentiable function satisfying

\begin{itemize}
\item for all $x\in\overline{\Omega_{0}}$, $X\left(  0,x\right)  =x$;

\item for all $t\geq0$, $X_{t}=X\left(  t,\cdot\right)  :\overline{\Omega_{0}%
}\rightarrow X\left(  t,\overline{\Omega_{0}}\right)  $ is a diffeomorphism; 
\end{itemize}
\end{itemize}
\end{definition}
Recall that $\Theta:C\rightarrow\mathbb{R}^{d}$ with $C\subset\mathbb{R}^{d}$ a
closed set is a diffeomorphism on its image if it is the restriction of some
$\tilde{\Theta}:O\rightarrow\mathbb{R}^{d}$ which is a diffeomorphism on its image
with $O$ open and $C\subset O$.

The image of $\overline{\Omega_{0}}$ through $X\left(  t,\cdot\right)  $,
$\overline{\Omega_{t}}:=X\left(  t,\overline{\Omega_{0}}\right)  =\left\{
X\left(  t,x\right)  :x\in\overline{\Omega_{0}}\right\}  $ represents the
configuration of the medium at time $t\geq0$. We denote by $\emph{CDM}\left(
\Omega_{0}\right)  $ the set of $X$ such that $(\Omega_0, X)$ is a continuous deformable medium.

\begin{remark}
In several places below, we shall write expressions such as  $X^{-1}$ or $\tilde X \circ X^{-1}$ for some  $X, \tilde X$ as in the definition. These slightly abusive  notations  will refer to invertibility or composition with respect to the space variable only, time being a parameter. For instance, for any $t$ and $x \in \overline{\Omega_{t}}$, $X^{-1}(t,x)$ refers to $X_t^{-1}(x)$, while $\tilde X \circ X^{-1}(t,x)$  refers to $\tilde{X}_t(X_t^{-1}(x))$. 
\end{remark}

\subparagraph{Lagrangian and eulerian standpoints.}

We have two different ways to think of any physical quantity $B$ that
characterizes the deformable medium $X\in\emph{CDM}\left(  \Omega_{0}\right)
$ :

\begin{itemize}
\item the lagrangian point of view: $B^{\ell}$ is attached to the particle
motion. In this setup $B^{\ell}\left(  t,x\right)  $ \textit{represents the
measure at time }$t$\textit{\ of the physical quantity associated to the
particle that leaves }$x$\textit{\ at time }$t=0$. It should be understood
automatically that the domain of definition for a "lagrangian quantity"
$B^{\ell}$ is  $\mathbb{R}_{+}\times\overline
{\Omega_{0}}$;

\item the eulerian point of view: $B^e$ is measured as time evolves, say between
$t_{1},t_{2}$ in a geometrical point $x\in%
{\textstyle\bigcap\nolimits_{t\in\left[  t_{1},t_{2}\right]  }}
X\left(  t,\overline{\Omega_{0}}\right)  $. 
 It should be understood automatically that the domain of definition of $B^{e} $ is $%
{\displaystyle\bigcup_{t\geq0}}
\{t\}\times X\left(  t,\overline{\Omega_{0}}\right)  $;
\end{itemize}

The two points of view are formally equivalent in the sense that any
\textquotedblleft lagrangian physical quantity\textquotedblright\ can be
regarded as an \textquotedblleft eulerian physical quantity\textquotedblright%
\ via the family of diffeomorphisms $\left(  X\left(  t,\cdot\right)  \right)
_{t\geq0}$. More precisely, given $B$ we have
\begin{equation}
B^{e}\left(  t,x\right)  =B^{\ell}\left(  t,X^{-1}\left(  t,x\right)  \right)
\text{ for all }t\geq0\text{ and }x\in X\left(  t,\overline{\Omega_{0}%
}\right)  , \label{Lagrange_to_Euler}%
\end{equation}
and conversely
\begin{equation}
B^{\ell}\left(  t,x\right)  =B^{e}\left(  t,X\left(  t,x\right)  \right)
\text{ for all }t\geq0\text{ and }x\in\overline{\Omega_{0}}.
\label{Euler_to_Lagrange}%
\end{equation}
As explained in \cite{BurteaGavrilyukPerrin2024}  ,
the Lagrangian point of view is the
natural framework to generalize the classical stationary action principle.
However, the resulting Lagrange equations are usually too complicated to handle whereas they are considerably simpler in eulerian coordinates.

Using the conventions established by relations $\left(
\text{\ref{Lagrange_to_Euler}}\right)  $ and $\left(
\text{\ref{Euler_to_Lagrange}}\right)  $ and using the chain rule for
derivatives we have the following relation.
\begin{equation}
\operatorname*{trace}\left(  D\psi^{\ell}(DX)^{-1}\right)  =\left(
(DX)^{-1}\right)  _{ji}\partial_{j}\left(  \psi^{\ell}\right)  ^{i}=\frac
{1}{\det DX}\operatorname{Cof}(DX)_{ij}\partial_{j}\left(  \psi^{\ell}\right)
^{i}=\left(  \operatorname{div}\psi^{e}\right)  ^{\ell}.
\label{div_non_conservativ}%
\end{equation}
Using this relation we infer that for any scalar field $B^{e}$ and for any
vector field $\psi^{e}$ with compact support in $\cup_{t \in \mathbb{R}_+} \{t\} \times \Omega_t$, 
\begin{align*}
\int_{\Omega_{0}}(\nabla B^{e})^{\ell}\cdot\psi^{\ell}\det DX\mathrm{d}x  &
=\int_{\Omega_t}\nabla B^{e}\cdot\psi^{e}\mathrm{d}%
x=-\int_{\Omega_t}B^{e}\operatorname{div}\psi
^{e}\mathrm{d}x\\
&  =-\int_{\Omega_{0}}B^{\ell}(\operatorname{div}\psi^{e})^{\ell}\det
DX\mathrm{d}x\\
&  =-\int_{\Omega_{0}}B^{\ell}\frac{1}{\det DX}\operatorname{Cof}%
(DX)_{ij}\partial_{j}\left(  \psi^{\ell}\right)  ^{i}\det DX\mathrm{d}x\\
&  =\int_{\Omega_{0}}\partial_{j}(\operatorname{Cof}(DX)_{ij}B^{\ell}%
)(\psi^{\ell})^{i}\mathrm{d}x
\end{align*}
From which we obtain that for any $i\in\overline{1,d}$%
\begin{equation} \label{gradient_euler_lagrange}
\frac{1}{\det DX}\partial_{j}\left(  \operatorname{Cof}(DX)_{ij}B^{\ell
}\right)  =\left(  \partial_{i}B^{e}\right)  ^{\ell},
\end{equation}
or in a compact form%
\[
\frac{1}{\det DX}\operatorname{div}\left(  \operatorname{Cof}(DX)B^{\ell
}\right)  =(\nabla B^{e})^{\ell}.
\]
In particular, taking $B=1$ we find the so-called \emph{Piola identity}
\begin{equation}
\operatorname{div}\left(  \operatorname{Cof}(DX)\right)  =0.
\end{equation}

\subsection{Description of immiscible incompressible-compressible
phases} \label{section_immiscible}

We consider $\Omega\subset\mathbb{R}^{d}$ a smooth open set which
represents the fixed domain of an immiscible two-phase flow made of a liquid (symbol : $+$) and a gas (symbol : $-$). We give ourselves $\Omega_{+,0},\Omega_{-,0}$ two smooth disjoint open subsets of $\Omega$, which model the region occupied initially by the liquid and the gas respectively. We assume that the gas bubbles are surrounded by the liquid, which translates into $\Omega = \Omega_{+, 0} \cup \overline{\Omega_{-,0}}$. 

In what follows, most of the quantities that will appear will be lagrangian ones, so that the upperscript $^\ell$ will be omitted to lighten notations. Only the upperscript $^e$ relative to eulerian quantities will be explicit.

We suppose that the liquid is an incompressible fluid with constant density
$\varrho_{+}$. The evolution of the liquid part will be determined by
$X_{+}\in\emph{CDM}\left(  \Omega_{+,0}\right)  $ along with a scalar field
$p_{+}:\Omega_{+,0}\rightarrow\mathbb{R}$ which can be seen as a constraint
that will enforce the incompressibility of the liquid namely $\det
DX_{+}\left(  t,x\right)  =1.$

The gas is assumed to be a compressible barotropic fluid with (lagrangian) variable density
$\varrho_{-}$. \ The evolution of the gas part will be determined by $X_{-}%
\in\emph{CDM}\left(  \Omega_{-,0}\right).$ From the conservation of mass of gas, we deduce  :%
\begin{equation} \label{rho-lagrange}
\varrho_{-}\left(  t,x\right)  =\frac{\varrho_{-,0}\left(  x\right)  }{\det
DX_{-}(t,x)},
\end{equation}
where $\varrho_{-,0}\left(  x\right)  $ is the initial density. We denote $p_{-} = p_-(s),e_{-} = e_-(s)$ the pressure and the  Helmholtz free energy of the gas. They
are known functions of a real variable $s$, and are linked by the relation%
\begin{equation}
p_{-}\left(  s\right)  =s^{2}e_{-}^{\prime}\left(  s\right)  .
\label{presure_energy}%
\end{equation}

We denote $\mathcal{I}(t)$ the interface between the immiscible liquid and gas at time $t$. At $t=0$, 
\[
\mathcal{I}\left(  0\right)  =\partial\Omega_{+,0}\cap\partial\Omega_{-,0} = \partial \Omega_{-,0}%
\]
(we remind that the gas bubbles are surrounded by the liquid one). For $t > 0$, we have 
\begin{equation}
\mathcal{I}\left(  t\right)  =X_{+}\left(  t,\mathcal{I}\left(  0\right)
\right)  =X_{-}\left(  t,\mathcal{I}\left(  0\right)  \right)  .
\label{relation_1}%
\end{equation}

The interpretation is as follows:  the two-phase fluid occupies the initial configuration
$\overline{\Omega_{0}}\subset\mathbb{R}^{n}$, each point $x\in\overline
{\Omega_{0}}\backslash\mathcal{I}\left(  0\right)  $ is occupied either by a
particle of liquid either by a particle of gas, particle which will change
positions as time evolves. At the level of the interface, each spatial
position $x\in\mathcal{I}\left(  0\right)  $ is occupied by two
particles\footnote{Particles should be understood as infinitesimal volumes
(fluid parcel) with caracteristic scale much larger than the mean distance
between two "real" fluid particules but neglectable w.r.t. the characteristic
size of the domain where the mixture evolves.} : one of liquid, the other one of
gas. If $x\in\overline{\Omega_{\pm,0}},$ then $X_{\pm}\left(  t,x\right)  $
represents the position at time $t$ of the particle that was located in $x$ at
time $t=0$. The fields $u_{\pm}\left(  t,x\right)  =\dot{X}_{\pm}\left(  t,x\right)
$ represent the velocities of the particles initially at $x$.

At the level of the interface, it should be noted that in general liquid and gas particles
located in $x\in\mathcal{I}\left(  0\right)  $ at $t=0$ will find themselves
at later time $t>0$ in different positions $X_{+}\left(  t,x\right)
,X_{-}\left(  t,x\right)  \in\mathcal{I}\left(  t\right)  $.
\begin{remark}
    In all what follows, the unit normal vector $n$ at the interface $\mathcal{I}(0)$ will be oriented from $\Omega_{-,0}$ to $\Omega_+,0$, that is from the gas to the liquid (and the same for $n^e_t$ at $\mathcal{I}(t)$). 
    \end{remark}
    We furthermore have the property: for $t > 0$
    \begin{equation} \label{constraint_fixed_bd}
    X_+(t,\pa \Omega) = \pa \Omega
    \end{equation}
    which expresses that liquid particles that are initially at the fixed boundary $\pa \Omega$ remain there at any time. 
\begin{proposition} \label{prop_normal_velocities}
Relation \eqref{relation_1} implies in eulerian coordinates  that
\begin{equation}
u_{+}^{e}(t,\cdot)\cdot n_t^{e}=u_{-}^{e}(t,\cdot)\cdot n_t^{e} = v_{\mathcal{I}}(t,\cdot) \text{ on }\mathcal{I}\left(
t\right)  \label{relation_normale}%
\end{equation}
with  $v_{\mathcal{I}}$ the normal velocity associated to the family  $\left(\mathcal{I}(t)\right)_{t \ge 0}$, see below. Moreover, the first equality reads in lagrangian coordinates: 
\begin{equation}
\dot{X}_{+}\circ X_{+}^{-1}\circ X_{-}\cdot\operatorname{Cof}(DX_{-}%
)n=\dot{X}_{-}\cdot\operatorname{Cof}(DX_{-})n \quad \text{ on }%
\mathcal{I}\left(  0\right). \label{relation_normale_lag}%
\end{equation}
\end{proposition}
\begin{proof}
We first remind the definition of the normal velocity $v_{\mathcal{I}}$ associated to $\left(\mathcal{I}(t)\right)_{t \ge 0}$, following \cite{Pruss}. Let $t \ge 0$, $x \in \mathcal{I}(t)$.  Let  $\tau \rightarrow X_{\mathcal{I}}(\tau)$ a smooth map from a neighborhood of $t$ to $\R^3$  such that 
$$ \forall \tau, \: X_{\mathcal{I}}(\tau) \in \mathcal{I}(\tau), \quad X_{\mathcal{I}}(t) = x $$
It can be easily shown that $\dot{X}_{\mathcal{I}}(t) \cdot n^e_t(x)$ depends only on $t$ and $x$, not on the choice of the path $X_{\mathcal{I}}$, see \cite{Pruss} for details, or Section \ref{sec_shape_constrained} below. This allows to define the normal velocity by the formula: 
$$ v_{\mathcal{I}}(t,x) \:  :=  \: \dot{X}_{\mathcal{I}}(t) \cdot n^e_t(x) $$
Now, from identity \eqref{relation_1}, we know that for all $x_0 \in \mathcal{I}(0)$,  $X_\pm(t,x_0) \in \mathcal{I}(t)$ for all $t \ge 0$. For fixed $t$ and $x = X_\pm(t,x_0)$, we can then take $X_{\mathcal{I}}(\tau) = X_\pm(t+\tau,x_0)$ to get
$$ v_{\mathcal{I}}(t,X_\pm(t,x_0)) =  \dot{X}_\pm(t,x_0) \cdot n_t^e(X_\pm(t,x_0)) = u^e_\pm(t,X_\pm(t,x_0)  \cdot n_t^e(X_\pm(t,x_0)) $$
Hence, in eulerian coordinates: for all $t \ge 0$,
$$v_{\mathcal{I}}(t,\cdot) = u^e_\pm(t,\cdot) \cdot n^e_t.  $$
This proves \eqref{relation_normale}. Finally, relation \eqref{relation_normale_lag} is just the translation in lagrangian coordinates of the first equality in \eqref{relation_normale}, remembering that for all $t \ge 0$, $x \in \mathcal{I}(0)$, 
$$ n^e_t(X_-(t,x)) = \frac{\left(DX^{-1}_-(t,X_-(t,x)\right)^T n(x)}{\|\left(DX^{-1}_-(t,X_-(t,x)\right)^T n(x)\|} = \frac{\operatorname{Cof}(DX_{-}(t,x)) n(x)}{\|\operatorname{Cof}(DX_{-}(t,x)) n(x)\|} 
$$

\end{proof}


\begin{remark}
Similarly, one can deduce from \eqref{constraint_fixed_bd} that $u^e_+ \cdot n\vert_{\partial \Omega} = 0$  in eulerian coordinates, or equivalently that $\dot{X}_{+} \cdot\operatorname{Cof}(DX_{+})n = 0  \text{ on }\partial \Omega$ in lagrangian coordinates. This follows from replacing $\mathcal{I}(0)$ by $\partial \Omega$, $X_-$ by $X_+$ and  $X_+$  by $\mathrm{Id}$ in Proposition \ref{prop_normal_velocities}. 
\end{remark}

\begin{remark}
\label{Remark_var_norm}By the same line of reasoning as in Proposition \ref{prop_normal_velocities}, if we are given two
one-parameter families $X_{\pm}^{\mu}$ depending in a differential manner of the parameter $\mu$, 
and such that
\[
\mathcal{I}^{\mu}\left(  t\right)  :=X_{+}^{\mu}\left(  t,\mathcal{I}(0)\right)  =X_{-}^{\mu}\left(  t,\mathcal{I}(0)\right)
\]
we have 
\begin{equation}
\frac{\partial X_{-}^{\mu}}{\partial\mu}\cdot\operatorname{Cof}(DX_{-}^{\mu
})n =\frac{\partial X_{+}^{\mu}}{\partial\mu}\circ (X^\mu_{+})^{-1}\circ
X^\mu_{-}\cdot\operatorname{Cof}(DX_{-}^{\mu})n\text{ on }\mathcal{I}(0). 
\label{var_norm_eq}%
\end{equation}
Similarly, if $X_{+}^{\mu}(\pa \Omega) = \pa \Omega$, we find that 
\begin{equation}
   \frac{\partial X_{+}^{\mu}}{\partial\mu} \cdot\operatorname{Cof}(DX^\mu_{+})n = 0  \text{ on }\partial \Omega.
\label{var_norm_eq_2}%
\end{equation}

\end{remark}

\medskip
Now that we have recovered the equalities \eqref{relation_normale} and \eqref{relation_normale_lag} on the normal velocities at the interface, we turn to the formal derivation of the remaining equations for our two-phase flow. They can be deduced from a least action principle. We define the action functional associated to the flow by
\begin{align}
&  \mathcal{A}\left(  t_{0},t_{1},X_{+},p_{+},X_{-}\right)  \nonumber\\
&  :=\int_{t_{0}}^{t_{1}}\int_{\Omega_{+,0}}\varrho_{+}\frac{|\dot{X}_{+}%
|^{2}}{2}+p_{+}\left(  \det DX_{+}-1\right)  +\int_{t_{0}}^{t_{1}}\int
_{\Omega_{-,0}}\varrho_{-,0}\frac{|\dot{X}_{-}|^{2}}{2} - \varrho_{-,0}%
e_{-}\left(  \frac{\varrho_{-,0}}{\det DX_{-}}\right)
.\label{action_two_fluids}%
\end{align}
with $X_\pm \in CDM(\Omega_{\pm,0})$ and $p_+ : \Omega_{+,0} \rightarrow \mathbb{R}$. We shall obtain the equations verified by $X_{+},p_{+},X_{-}$ by asking that they
are critical points of $\mathcal{A}\left(  t_{0},t_{1},X_{+},p_{+}%
,X_{-}\right)  $ for any $t_{0}$ and $t_{1}$. More precisely, the evolution of
$X_{+},p_{+},X_{-}$ will be characterised by the Euler-Lagrange equations associated to $\mathcal{A}$, restricted to the class of diffeomorphisms $X_\pm$ that obey
relation $\left(  \text{\ref{relation_1}}\right)  $, that preserve the external boundary $\partial \Omega$, and that have prescribed initial and final conditions. Concretely, we consider a one-parameter family  $\left(  X^\mu_{+},p^\mu_{+},X^\mu_{-}\right)$, $\mu \in (-\mu_0, \mu_0)$ for some $\mu_0 > 0$, with the following properties: for all $\mu$,    
\begin{align}
& X_{+}^{\mu}\left(  t,\mathcal{I}\left(  0\right)  \right)     =X_{-}^{\mu
}\left(  t,\mathcal{I}\left(  0\right)  \right)  ,\label{constraint}\\
& X_{\pm}^{\mu}\left(t,\partial\Omega \right)   = \partial\Omega,
\label{constraint-2} \\
& X^\mu_\pm(t_0,\cdot)  \text{ and }  X^\mu_\pm(t_1,\cdot)  \text{  are constant in $\mu$.} \label{constraint-3} 
\end{align}
We denote: 
$$\left. \left(  \delta X_{+},\delta p_{+},\delta X_{-}\right) = \frac{d}{d\mu} \left(  X^\mu_{+}, p^\mu_{+},X^\mu_{-}\right) \right\vert_{\mu=0}   $$ 
The Euler Lagrange equations are recovered by stating that
\[
\left.  \frac{d}{d\mu}\mathcal{A}\left(  t_{0},t_{1},X_{+}^{\mu},p_{+}^{\mu
},X_{-}^{\mu}\right)  \right\vert _{\mu=0}=0.
\]


Of course, we have that%
\begin{align*}
&  \left.  \frac{d}{d\mu}\mathcal{A}\left(  t_{0},t_{1},X_{+}^{\mu},p_{+}%
^{\mu},X_{-}^{\mu}\right)  \right\vert _{\mu=0}\\
&  =\int_{t_{0}}^{t_{1}}\int_{\Omega_{+,0}}\varrho_{+}\dot{X}_{+}^{j}%
\delta\dot{X}_{+}^{j}+p_{+}\left(  \operatorname{Cof}DX_{+}\right)
_{jk}\partial_{k}(\delta X_{+}^{j})+\int_{t_{0}}^{t_{1}}\int_{\Omega_{+,0}%
}\left(  \det DX_{+}-1\right)  \delta p_{+}\\
&  + \int_{t_{0}}^{t_{1}}\int_{\Omega_{-,0}}\varrho_{-,0}\dot{X}_{-}^{j}%
\delta\dot{X}_{-}^{j} + \frac{\varrho_{-,0}^{2}}{(\det DX_{-})^{2}}e_{-}%
^{\prime}\left(  \frac{\varrho_{-,0}}{\det DX_{-}}\right)  \left(
\operatorname{Cof}DX_{-}\right)  _{jk}\partial_{k}(\delta X_{-}^{j}).
\end{align*}
Performing integration by parts and taking into account $\left(
\text{\ref{constraint}}\right)  $-$\left(  \text{\ref{constraint-2}}\right)  $-$\left(  \text{\ref{constraint-3}}\right)  $
as well as  $\left(  \text{\ref{presure_energy}}\right)  $ and \eqref{var_norm_eq_2}, we obtain that%
\begin{align}
&  \left.  \frac{d}{d\mu}\mathcal{A}\left(  t_{0},t_{1},X_{+}^{\mu},p_{+}%
^{\mu},X_{-}^{\mu}\right)  \right\vert _{\mu=0}\nonumber\\
&  =-\int_{t_{0}}^{t_{1}}\int_{\Omega_{+,0}}\varrho_{+}\ddot{X}_{+}^{j}\delta
X_{+}^{j}+\partial_{k}(p_{+}\left(  \operatorname{Cof}DX_{+}\right)
_{jk})(\delta X_{+}^{j})\nonumber\\
&  +\int_{t_{0}}^{t_{1}}\int_{\Omega_{+,0}}\left(  \det DX_{+}-1\right)
\delta p_{+}\nonumber\\
&  -\int_{t_{0}}^{t_{1}}\int_{\Omega_{-,0}}\varrho_{-,0}\ddot{X}_{-}^{j}\delta
X_{-}^{j}+\partial_{k}\left(  p_{-}\left(  \frac{\varrho_{-,0}}{\det DX_{-}%
}\right)  \left(  \operatorname{Cof}DX_{-}\right)  _{jk}\right)  \delta
X_{-}^{j}\nonumber\\
&  -\int_{t_{0}}^{t_{1}}\int_{\mathcal{I}\left(  0\right)  }%
p_{+}\operatorname{Cof}(DX_{+})_{jk}\delta X_{+}^{j}n^{k}\mathrm{d}\sigma
+\int_{t_{0}}^{t_{1}}\int_{\mathcal{I}\left(  0\right)  }p_{-}\left(
\frac{\varrho_{-,0}}{\det DX_{-}}\right)  \operatorname{Cof}(DX_{-}%
)_{jk}\delta X_{-}^{j}n^{k}\mathrm{d}\sigma.\label{variation}%
\end{align}
From  the change of variable formula for surface integrals, {\it cf.} \cite[Proposition 5.4.3]{Henrot}, we have for any function $f$:  
$$ \int_{\mathcal{I}(0)} f =   \int_{X_{+}^{-1}\circ X_{-}(\mathcal{I}(0))} f = \int_{\mathcal{I}(0)} f \circ X_{+}^{-1}\circ X_{-} \, \| \operatorname{Cof}(D(X_{+}^{-1}\circ X_{-})) n \| $$
Together with the change of variable formula for the normal vector : 
$$ n \circ X_+^{-1} \circ X_- =  \frac{\operatorname{Cof}(D(X_{+}^{-1}\circ X_{-}))n}{\|\operatorname{Cof}(D(X_{+}^{-1}\circ X_{-}))n\|} =   \frac{\operatorname{Cof}(DX_{+}^{-1})\circ X_{-} \, \operatorname{Cof}(DX_-)n}{\|\operatorname{Cof}(D(X_{+}^{-1}\circ X_{-}))n\|}  $$
 it implies 
\begin{multline*}
\int_{t_{0}}^{t_{1}}\int_{\mathcal{I}\left(  0\right)  }p_{+}\left(
\operatorname{Cof}DX_{+}\right)  _{jk}\delta X_{+}^{j}n^{k}d\sigma
\\ 
=\int_{t_{0}}^{t_{1}}\int_{\mathcal{I}\left(  0\right)  }\left[  p_{+}\circ
X_{+}^{-1}\circ X_{-}\right]  \left(  \operatorname{Cof}DX_{-}\right)
_{jk}\left[  \delta X_{+}^{j}\circ X_{+}^{-1}\circ X_{-}\right]  n%
^{k}d\sigma
\end{multline*} 
Taking into account \eqref{var_norm_eq} in  Remark \ref{Remark_var_norm}, the whole boundary term appearing in the derivative of the action can be written as%
\begin{equation} 
\int_{t_{0}}^{t_{1}}\int_{\partial\mathcal{I}\left(  0\right)  }\left\{
p_{-}\left(  \frac{\varrho_{-,0}}{\det DX_{-}}\right)  -p_{+}\circ X_{+}%
^{-1}\circ X_{-}\right\}  \delta X_{-}\cdot\operatorname{Cof}(DX_{-}%
)n\mathrm{d}\sigma\label{var_integral_term}%
\end{equation}
If we first take the variation $\partial X_\pm$ compactly supported inside the
domains $\pm$, we obtain the equations%
\[
\varrho_{+}\ddot{X}_{+}^{j}+\partial_{k}\left(  p_{+}\left(
\operatorname{Cof}DX_{+}\right)  _{jk}\right)  =0\text{ in }\Omega_{+,0},
\]
respectively
\[
\varrho_{-,0}\left(  x\right)  \ddot{X}_{-}^{j}+\partial_{k}\left(  p_{-}\left(
\frac{\varrho_{-,0}}{\det DX_{-}}\right)  \left(  \operatorname{Cof}%
DX_{-}\right)  _{jk}\right)  =0\text{ in }\Omega_{-,0},
\]
for all $j=1\dots d$. Using \eqref{gradient_euler_lagrange} and the fact that $\det(DX_+) = 1$, the first equation translates in eulerian coordinates into \footnote{Recall that for the liquid $\varrho_{+}=cst$ it is for this
reason that we do not mark it with an $e$-upperscript.}%
\[
\varrho_{+}(\partial_{t}u_{+}^{e}+u_{+}^{e}\cdot\nabla u_{+}^{e})+\nabla
p_{+}^{e}=0\text{ in }\Omega_{+,t}:=X_{+}\left(  t,\Omega_{+,0}\right)
\]
By \eqref{gradient_euler_lagrange}, the second equation together with \eqref{rho-lagrange} yields
\[
\left\{
\begin{array}
[c]{l}%
\partial_{t}\varrho_{-}^{e}+\operatorname{div}\left(  \varrho_{-}^{e}u_{-}%
^{e}\right)  =0\text{ in }\Omega_{-,t}:=X_{-}\left(  t,\Omega_{-,0}\right)  \\
\varrho_{-}^{e}(\partial_{t}u_{-}^{e}+u_{-}^{e}\cdot\nabla u_{-}^{e})+\nabla
p_{-}\left(  \varrho_{-}^{e}\right)  =0\text{ in }\Omega_{-,t}.
\end{array}
\right.
\]

Note that the variation $\partial X_-$, and more generally $X_-^\mu$   can be taken arbitrarily, and then 
$X_{+}^\mu$ is chosen to ensure the validity of $\left(  \text{\ref{constraint}}\right)
$-$\left(  \text{\ref{constraint-2}}\right)$-$\left(  \text{\ref{constraint-3}}\right)$. Hence, from $\left(
\text{\ref{var_integral_term}}\right)$, we deduce
\[
p_{-}\left(  \frac{\varrho_{-}}{\det DX_{-}}\right)  =p_{+}\circ X_{+}%
^{-1}\circ X_{-}\text{ on }\mathcal{I}\left(  0\right)
\]
or in eulerian coordinates%
\[
p_{-}\left(  \varrho_{-}^{e}\left(  t,x\right)  \right)  =p_{+}^{e}\left(
t,x\right)  \text{ on }\mathcal{I}\left(  t\right)  .
\]


Eventually, this lagrangian approach allows to recover all equations in eulerian coordinates
(and since by now there is no danger of confusion we drop the $e$-upperscript)%
\begin{equation} 
\left\{
\begin{array}
[c]{l}%
\Omega = \Omega_{+,t} \sqcup \Omega_{-,t} \sqcup \mathcal{I}(t), \quad \Omega_{-,t} \Subset \Omega, \\
\varrho_{+}(\partial_{t}u_{+}+u_{+}\cdot\nabla u_{+})+\nabla p_{+}=0\text{ in
} \Omega_{+,t},\\
\partial_{t}\varrho_{-}+\operatorname{div}\left(  \varrho_{-}u_{-}\right)
=0\text{ in }\Omega_{-,t},\\
\varrho_{-}(\partial_{t}u_{-}+u_{-}\cdot\nabla u_{-})+\nabla p_{-}\left(
\varrho_{-}\right)  =0,\text{ in }\Omega_{-,t},\\
u_+ \cdot n = 0 \text{ on } \pa \Omega, \\
u_{+}\cdot n=u_{-}\cdot n = v_{\mathcal{I}}\text{ on }\mathcal{I}\left(  t\right), \\ 
p_{+}=p_{-}\left(  \varrho_{-}\right)  ,\text{ on }\mathcal{I}\left(
t\right) .
\end{array}
\right.  \label{system_1}%
\end{equation}

\begin{remark}
To take surface tension at the interface into account, the action
functional $\left(  \text{\ref{action_two_fluids}}\right)  $ should be
modified with the extra term%
\[
\int_{t_{0}}^{t_{1}}\int_{X_{-}(t,\mathcal{I}(0))}\mathrm{d}\sigma.
\]
 The derivative of the function%
\[
\mu\rightarrow\int_{t_{0}}^{t_{1}}\int_{X_{-}^{\mu}(t,\mathcal{I}(0)%
)}\mathrm{d}\sigma
\]
with $\left(  X_{-}^{\mu}\right)  _{\mu}$ virtual motions, computed at $\mu=0$
gives%
\[
\int_{t_{0}}^{t_{1}}\int_{\mathcal{I}(0)} H(X_-) \operatorname{Cof}(DX_{-}) n \cdot \delta
X_{-}\mathrm{d}\sigma,
\]
where $H(t,\cdot)$ is the mean curvature of the interface $\mathcal{I}(t)$. The impact on the Euler-Lagrange
equations is that the first boundary condition from $\left(
\text{\ref{system_1}}\right)  $ changes into%
\[
p_{+}-p_{-}\left(  \varrho_{-}\right)  =H.
\]
\end{remark}

\begin{remark} \label{rem_compatibility condition}
    The formal derivation  we have just performed consists in showing that a critical point of $\mathcal{A}$ satisfies system \eqref{system_1} in eulerian coordinates. It is motivated by the least action principle, whereby the evolution in time of the two-phase fluid should be governed by a minimizer of the action $\mathcal{A}$, therefore by a critical point of $\mathcal{A}$, and therefore by a solution of  \eqref{system_1}. Still, wether or not such  a minimizer exists, or even a critical point of $\mathcal{A}$ or a solution of \eqref{system_1}, is an independent question that requires a specific study. It notably involves some compatibility conditions on the initial data, due to incompressibility of the field $u_+$: a natural one is that $\int_{\mathcal{I}(0)} u_+\vert_{t=0} \cdot n = 0$. Indeed, integrating the divergence-free condition for $u_+$ over $\Omega_{t,+} := X_+(t,\Omega_{0,+})$ implies that $\int_{\mathcal{I}(t)} u_+ \cdot n = 0$ for all time. Such kind of requirements on the initial data and the well-posedness problem will reappear in Section \ref{sec_WP_reduced}.  \end{remark}
\begin{remark}
  
We mentioned in the introduction several works where the derivation of PDEs from fluid mechanics was obtained from a similar Euler-Lagrange approach. Among those, one may stress again article \cite{GlassSueur2012} on the derivation of fluid-solid interaction systems. In this case, the flow $X_-$ inside each solid particle is looked for in the class of affine direct isometries, so that in eulerian coordinates, denoting $\Omega_k(t)$ the domain of the  $k$-th particle at time $t$:  
$$ u_-(t,x) = c_k'(t) + \omega_k(t) \times (x - c_k(t)), \quad x \in \Omega_k(t).$$
In this setting, the  action functional is reduced to the kinetic energy: 
\begin{align*}
&  \mathcal{A}\left(  t_{0},t_{1},X_{+},p_{+},X_{-}\right)  \nonumber\\
&  :=\int_{t_{0}}^{t_{1}}\int_{\Omega_{+,0}}\varrho_{+}\frac{|\dot{X}_{+}%
|^{2}}{2}+p_{+}\left(  \det DX_{+}-1\right)  +\int_{t_{0}}^{t_{1}}\int
_{\Omega_{-,0}}\varrho_{-,0}\frac{|\dot{X}_{-}|^{2}}{2}
\end{align*}
Moreover, the variation $\pa X_-$ is no longer arbitrary, as $X_-^\mu$ has to be an isometry for each $\mu$ in each $\Omega_k(0)$.  If $X^\mu_-(t,x) = c^\mu_k(t) + R^\mu_k(t) x $, $x \in \Omega_k(0)$, with $c^\mu_k(t) \in \R^3$ and $R^\mu_k(t) \in SO_3(\R)$,  it  can only be of the form
$$\pa X_-(t,x) = \frac{d}{d\mu} c_k^\mu(t)\vert_{\mu=0}  +  \frac{d}{d\mu} R_k^\mu(t)\vert_{\mu=0} x, \quad x \in \Omega_k(0),$$
 with $\mu \rightarrow c^\mu_k$ a path in $\R^3$ through $c_k$  and $\mu \rightarrow R^\mu_k$ a path in $SO3(\R)$ through $R_k$. It is classical that this is equivalent to 
 $$  \pa X_-(t,x)  = \tilde{v}_k(t) + \tilde{\omega}_k(t) \times (R_k(t)x), \quad \tilde{v}_k,  \omega_k \text{ with values in $\R^3$}, \quad x \in \Omega_k(0). $$
 In this reduced setting, it can be shown that the continuity of the pressure at the interface is replaced by the following relations: 
\begin{align*}
     M_k c_k''(t) &  = - \int_{\pa \Omega_k(t)} p_+(t,\cdot) n_t, \\ 
    \frac{d}{dt} (J_k \omega_k)(t)  & = \int_{\pa \Omega_k(t)}  p_+(t,\cdot)  n_t \times (x - c_k(t)) \\
    \text{ with } \: J_k(t) & = \int_{\Omega_k(t)} \Big( |x - c_k(t)|^2 -  (x-c_k(t)) \otimes (x-c_k(t)) \Big). 
\end{align*} 
We refer to \cite{GlassSueur2012} for much more. 
\end{remark}

\section{Reduced models with shape-constrained dynamics}
\label{sec_shape_constrained}
We turn in this section to the derivation of reduced models for bubbly flows, where we wish to simplify the geometry and dynamics of the bubbles. One crude and very common simplification that we have in mind is to assimilate bubbles to spheres. With the notations of the previous section, the idea is to assume that initially the interface is 
$$ \mathcal{I}(0) = \cup_{i=1}^N S(c_{i,0}, r_{i,0})$$
that is a union of disjoint spheres of centers $c_{i,0}$ and radii $r_{i,0}$ : $\forall i,j$, $|c_{i,0} - c_{j,0}| > r_{i,0} + r_{j,0}$. We would like the interface at positive times to still be described by a union of spheres:
$$ \mathcal{I}(t) = \cup_{i=1}^N S(c_i(t), r_i(t))$$
and to determine the dynamics of the centers $c_i(t)$ and radii $r_i(t)$. Of course, adding such a constraint on the interface to  \eqref{system_1} would lead to an overdetermined system. The point is to understand how to relax \eqref{system_1} in order to obtain a well-posed system when incorporating the constraint of sphericity.  As discussed in the introduction, to restrict to spherical bubbles may be too stringent:  consideration of ellipsoidal bubbles, or of more general shapes (still with a limited number of degrees of freedom) may be more accurate. Therefore, we introduce in the next section a general formalism that leads  to reduced models for a large class of bubble shapes. 

\subsection{General formalism} \label{sec_formalism}
We shall represent our bubbles thanks to a family $(\mathcal{S}_m)_{m \in M}$, where for each $m \in M$,  $\mathcal{S}_m$ is the boundary of a smooth bounded domain $\Omega_m$ of $\R^3$. The parameter set $M$ encodes the shape constraints that we impose on the bubbles: it is finite dimensional, representing the finite number of degrees of freedom allowed for the bubbles deformation. More precisely, we assume that $M$ is a smooth submanifold of $\R^d$ for some $d \ge 1$,  and that   
$$ \bigsqcup_{m \in M} \mathcal{S}_m := \bigcup_{m \in M} \{m\} \times \mathcal{S}_m  \subset M \times \R^3 $$
is  a smooth submanifold of $\R^d \times \R^3$. We further assume that $m \rightarrow S_m$ is one to one.  

\medskip
The simplest example we have in mind is spherical bubbles. For any $c \in \R^3$ and $r > 0$, we denote $S(c,r)$ the sphere of center $c$ and radius $r$. In this case we take  
$$ M = \Big\{ m = (c,r) \in  \R^3 \times (0,\infty),  S(c,r) \subset \Omega \big\}, \quad \mathcal{S}_m = S(c,r)  $$
Another example is ellipsoidal bubbles. For $c \in \R^3$, $S \in S_3^{++}(\R)$ the set of positive definite symmetric matrices, we define the ellipsoid
$$ E(c,S) = \{ c + Sy, \quad y \in \mathbb{S}^2 \}. $$
In that case we consider 
$$ M = \Big\{ m = (c,S) \in  \R^3 \times S_3^{++}(\R),  E(c,S) \subset \Omega \big\},, \quad  \mathcal{S}_m = E(c,S).    $$
\begin{remark}
We could easily extend our modelling to situations where the bubbles have different shapes, for instance some of them being spherical, some of them being ellipsoidal. This means we could consider different collections $\big(\mathcal{S}_m^k\big)_{m \in M_k}$, $k=1\dots N$, with bubble $1$ from the first collection  $\big(\mathcal{S}_m^1\big)_{m \in M_1}$, bubble $2$ from the second collection $\big(\mathcal{S}_m^2\big)_{m \in M_2}$, etc. 
\end{remark}

\begin{remark}
In all what follows, notation $\dot{m}$ will denote a tangent vector at $m$, that is an element of the tangent space $T_m M$. On the contrary, for any path $m = m(\tau)$ on $M$ parametrized by some $\tau$, the derivative with respect to the parameter will be denoted $m'$ to avoid confusion.  
\end{remark}
\medskip
A crucial notion  for what will follow is the notion of {\em normal velocity}, that we have briefly seen in the proof of Proposition \ref{prop_normal_velocities}, and that is described in detail in \cite{Pruss}. We adapt here this description to our setting of shape-constrained dynamics.  Let $m \in M$, $\dot{m} \in T_m M$  and $x \in \mathcal{S}_m$. We consider  a $C^1$ path in 
$\: \bigsqcup_{m \in M} \mathcal{S}_m$ through $(m,x)$: 
$$ (-\tau_0, \tau_0) \rightarrow \bigsqcup_{m \in M} \mathcal{S}_m, \quad  \tau \rightarrow  \big(M(\tau), X(\tau)\big)$$
 with  $M(0) = m$, $M'(0) = \dot{m}$ and $X(0) = x$. By the arguments in \cite[pp 76-77]{Pruss}\footnote{Those arguments are applied to interfaces parametrized by  $t \in \R$  but extend straightforwardly to parametrization by $m \in M$.} there exists a neighborhood of $m$ in $M$ such that for all $\tilde m$ in this neighborhood: 
 $$ \mathcal{S}_{\tilde m} \: = \: \Big\{ y + \rho(\tilde m,y) n(y), \: y \in \mathcal{S}_m \Big\} $$ 
 where $n(y)$ is the  normal vector to $\mathcal{S}_m$ at $y$, and $\rho = \rho(\tilde m,y)$ is a smooth real-valued function. For $\tau$ small enough so that $M(\tau)$ belongs to this neighborhood, as $X(\tau) \in \mathcal{S}_{M(\tau)}$, we can write
 $$ X(\tau) = y(\tau) +  \rho(M(\tau),y(\tau)) n(y(\tau)) $$
 for some function $\tau \rightarrow y(\tau)$ with values in $\mathcal{S}_m$ and  as regular as $X$. We may then compute the derivative at $\tau = 0$: 
 \begin{align*} 
 X'(0)  & = y'(0) + \langle \pa_m\rho(m,x) , \dot{m} \rangle n(x)  +  \langle \pa_y\rho(m,x) , y'(0)  \rangle n(x) +  \rho(m,x) \frac{d}{d\tau} \big[n(y(\tau)\big]\vert_{\tau=0} \\  
 & = \langle \pa_m\rho(m,x) , \dot{m} \rangle n(x) + 
 \frac{d}{d\tau} \Big[ y(\tau) +  \rho(m,y(\tau)) n(y(\tau))  \Big]\vert_{\tau=0} \\
 & = \langle \pa_m\rho(m,x) , \dot{m} \rangle n(x) + y'(0)  
 \end{align*}
 where we noticed for the last equality that $\rho(m,y) = 0$ for all $y \in \mathcal{S}_m$. As $y'(0)$ is tangent to $\mathcal{S}_m$, we end up with the normal velocity: 
$$ X'(0) \cdot n(x) = \langle \pa_m\rho(m,x) , \dot{m} \rangle   $$
This formula shows that this normal velocity depends only on  $m,x$ and $\dot{m}$, not on the full path $X$. It can be interpreted as  the normal velocity at $x$ induced by the infinitesimal displacement of the interface $\mathcal{S}_m$ with velocity $\dot{m}$. Finally, as it is linear in $\dot{m}$, we can introduce the map 
\begin{equation}  \label{defi_normal_velocity}
V :  (m,x) \in \bigsqcup_{m \in M} \mathcal{S}_m \rightarrow V_m(x) \in T^*_m M, \quad \big\langle V_m(x) , \dot{m} \big\rangle = X'(0) \cdot n(x)   
\end{equation}
for any path  $\tau \rightarrow \big(M(\tau),X(\tau)\big)$  in $\bigsqcup_{m \in M} \mathcal{S}_m$ such that $M(0) = m$, $M'(0) = \dot{m}$ and $X(0) =x$.

\medskip
The idea is now to apply the same least action principle as in Section \ref{section_immiscible}. More precisely, we  consider critical points $(X_+,p_+,X_-)$ of the lagrangian $\mathcal{A}$ in \eqref{action_two_fluids}, where we add to the constraints \eqref{relation_1}-\eqref{constraint_fixed_bd} the extra requirement that the interface $\mathcal{I}(t)$ is a union of $N$ disjoint bubbles from the class $(\mathcal{S}_m)_{m \in M}$: for some functions of time $m_1, \dots, m_N$,    
\begin{equation} \label{constraint_shape}
\forall t \ge 0, \quad \mathcal{I}(t) = \bigcup_{k=1}^N \mathcal{I}_k(t), \quad \mathcal{I}_k(t) = \mathcal{S}_{m_k}(t) \quad   \forall k=1, \dots, N. 
\end{equation}
Note that functions $m_k$ are uniquely defined as we assumed the map $m \rightarrow \mathcal{S}_m$ to be one-to-one.  
\begin{remark} \label{rem_normal_velocity}
    Condition \eqref{relation_1}, or \eqref{relation_normale} in eulerian formulation, can be made  more explicit under the extra requirement \eqref{constraint_shape}. Indeed, for all $k=1 \dots N$, for all $x \in \mathcal{I}_k(0)$,  if $\displaystyle X_-(0,x) \in \mathcal{I}_k(0) = S_{m_k(0)}$, then  $X_-(\tau,x) \in \mathcal{I}(\tau)= \mathcal{S}_{m_k(\tau)}$ for all $\tau  \ge 0$. Therefore, for any arbitrary fixed $t$, the map 
    $$ \big(M(\tau), X(\tau)\big) =  \big(m_k(t+\tau), X_-(t+\tau,x)\big)$$
    is a path in $\bigsqcup_{m \in M} \mathcal{S}_m$ such that $M(0) = m_k(t)$, $M'(0) = m_k'(t)$ and $X(0) = X_-(t,x)$. By the definition in \eqref{defi_normal_velocity} we have 
    $$ X'_-(t,x) \cdot n^e_t(X_-(t,x))  = X'(0)\cdot n^e_t(X_-(t,x))  =    \big\langle V_{m_k(t)}(X_-(t,x)) \: , \: m_k'(t) \big\rangle    $$ 
In other words, condition \eqref{relation_normale} can be refined into
\begin{equation} \label{relation_normale_refined}
u_{+}^{e}(t,\cdot)\cdot n_t^{e}=u_{-}^{e}(t,\cdot)\cdot n_t^{e} = \big\langle V_{m_k(t)}(\cdot) , m'_k(t) \big\rangle \quad  \text{ on } \quad \mathcal{I}\left(
t\right) = \mathcal{S}_{m_k(t)}, \quad \forall k=1\dots N.
\end{equation}
\end{remark}
At the level of the virtual motions $X_\pm^\mu, p_+^\mu$, \eqref{constraint_shape} further implies that besides \eqref{constraint}-\eqref{constraint-2}-\eqref{constraint-3} we must have 
\begin{equation} \label{constraint-4}
 X_\pm^\mu(t,\mathcal{I}_k(0)) = \mathcal{S}_{m^\mu_k(t)}, \quad \textrm{for some $m_k^\mu$ with values in $M$, for all $k=1\dots N$, for all $t \ge 0$.} 
\end{equation}
We introduce
$$m_k := m_k^\mu\vert_{\mu=0}, \quad  \pa m_k := \pa_\mu m_k^\mu\vert_{\mu=0}, \quad \pa X_\pm := \pa_\mu X_\pm^\mu\vert_{\mu=0}.$$
This time, for any fixed $x \in \mathcal{I}_k(0)$ and any fixed $t$, we consider 
$$ \big(M(\mu) , \mathcal{X}_\pm(\mu)\big) =   \left(m_k^\mu(t) ,  X_\pm^\mu(t,x) \right)$$
as a path in $\bigsqcup_{m \in M} \mathcal{S}_m$  such that $M(0) = m_k(t)$, $M'(0) = \pa m_k(t)$ and $\mathcal{X}_\pm(0) =X_\pm(t,x)$.  By the definition \eqref{defi_normal_velocity}, it follows that    
\begin{equation} \label{extra_constraint_variation}
\pa X_{\pm}(t,x) \cdot n^e_t(X_\pm(t,x))  =  \mathcal{X}_\pm'(0) \cdot n^e_t(X_\pm(t,x)) =  \big\langle V_{m_k}(X_\pm(t,x))  , \pa m_k(t) \big\rangle    
\end{equation}
where $n^e_t$ is the normal field at $\mathcal{I}_k(t) = \mathcal{S}_{m_k(t)}$. 

\medskip
It remains to derive the Euler-Lagrange equations satisfied by a critical point $(X_+,p_+,X_-)$ of the action $\mathcal{A}$. Most of the computations carried in Section \ref{section_immiscible} remain valid under the additional constraint on the shape of the bubbles. One can take for instance variations $\pa X_\pm$ and $\pa p_+$ that have compact support in time and space, and recover the incompressible Euler equation in   $\Omega_{+,t}$, resp. the compressible Euler equations in $\Omega_{-,t}$. As seen in Remark \ref{rem_normal_velocity}, the condition on the normal velocities takes the form \eqref{relation_normale_refined}. The main point is to understand how the continuity of the pressure at the interface is modified by the shape constraint. Exactly as in Section \ref{section_immiscible}, we obtain the relation \eqref{var_integral_term}, that we write as: for all $k=1 \dots N$,   
$$ \int_{t_0}^{t_1} \int_{\mathcal{I}_k(0)} \big(p_-(\varrho^e_- \circ X_-)  - p^e_+ \circ X_- \big) \pa X_- \cdot n^e(X_-) \| \operatorname{Cof}(DX_{-}%
)n \| \mathrm{d}\sigma  = 0   $$ 
The change compared to  Section \ref{section_immiscible} is the additional constraint \eqref{extra_constraint_variation}. Replacing in the previous equality (with minus sign), we find that 
\begin{equation*}
   \int_{t_0}^{t_1} \int_{\mathcal{I}_k(0)} \big(p_-(\varrho^e_- \circ X_-)  - p^e_+ \circ X_- \big) \big\langle V_{m_k}(X_-) , \pa m_k \big\rangle \| \operatorname{Cof}(DX_{-}%
)n \| \mathrm{d}\sigma  = 0   
\end{equation*}
where we remind that $m_k = m_k(t)$ is such that $\mathcal{I}(t) = \mathcal{S}_{m_k}(t)$, and $\pa m_k = \pa m_k(t) \in T_{m_k(t)}M$. As virtual motions may lead to arbitrary elements $\pa m_k \in T_{m_k} M$, we deduce that for all $t \ge 0$, for all $k=1\dots N$,
\begin{equation*}
   \int_{\mathcal{I}_k(0)} \big(p_-(\varrho^e_- \circ X_-)  - p^e_+ \circ X_- \big)  V_{m_k}(X_-) \| \operatorname{Cof}(DX_{-}%
)n \| \mathrm{d}\sigma  = 0   
\end{equation*}
or after a change of variable: for all $t \ge 0$, , for all $k=1 \dots N$,
\begin{equation}
   \int_{\mathcal{I}_k(t)} \big(p_-(\varrho^e_-(t,\cdot))   - p^e_+(t,\cdot) \big)  V_{m_k}(\cdot)   \mathrm{d}\sigma  = 0, \quad \mathcal{I}_k(t) = \mathcal{S}_{m_k(t)}   
\end{equation}
We insist that for each $t$ and $x \in \mathcal{I}_k(t)$, $ V_{m_k(t)}(x)$ is a linear  form on $T_{m_k(t)} M$, and so is the left-hand side. Eventually, the system \eqref{system_1} is replaced by  
\begin{equation} 
\left\{
\begin{array}
[c]{l}%
\Omega = \Omega_{+,t} \sqcup \Omega_{-,t} \sqcup \mathcal{I}(t), \quad \Omega_{-,t} \Subset \Omega, \\
\mathcal{I}(t) = \bigcup_{k=1}^N \mathcal{I}_k(t), \quad \mathcal{I}_k(t) =  \mathcal{S}_{m_k(t)}  \quad \forall  k, \\
\varrho_{+}(\partial_{t}u_{+}+u_{+}\cdot\nabla u_{+})+\nabla p_{+}=0\text{ in
}\Omega_{+,t},\\
\operatorname{div} u_{+} =0\text{ in
}\Omega_{+,t},\\
\partial_{t}\varrho_{-}+\operatorname{div}\left(  \varrho_{-}u_{-}\right)
=0\text{ in }\Omega_{-,t},\\
\varrho_{-}(\partial_{t}u_{-}+u_{-}\cdot\nabla u_{-})+\nabla p_{-}\left(
\varrho_{-}\right)  =0,\text{ in }\Omega_{-,t},\\
u_+ \cdot n = 0 \text{ on }\pa \Omega, \\
u_{+}\cdot n=u_{-}\cdot n = \big\langle V_{m_k(t)}(\cdot)  ,  m'_k(t) \big\rangle \text{ on }\mathcal{I}_k\left(  t\right) \quad \forall k,  \\
  \int_{\mathcal{I}_k(t)} \big(p_-(\varrho_-(t,\cdot))   - p^e_+(t,\cdot) \big) V_{m_k(t)}(\cdot)  \mathrm{d}\sigma  = 0 \quad \forall k.  
\end{array}
\right.  \label{system_2}%
\end{equation}

In the next paragraphs, we will consider special cases (spherical bubbles, ellipsoidal bubbles) for which   system \eqref{system_2}, especially the last boundary condition, will be made clearer. 

\subsection{Spherical gas bubbles} \label{sec_spherical_bubbles}

\subsubsection{Coupled incompressible/compressible model}
In this paragraph, we want to model the evolution of $N$ bubbles, assuming the sphericity of each bubble at all times. We will rely on the formalism of the previous paragraph. We wish the liquid/bubbles interface to be described for all time by 
$$ \mathcal{I}(t) = \bigcup_{k=1}^N \mathcal{I}_k(t), \quad \mathcal{I}_k(t) = S(c_k(t),r_k(t))$$
where $S(c,r)$ is the sphere of center $c$ and radius $r$. 
This shape constraint enters the framework of the previous paragraph. We set 
$$ M = \Big\{ m = (c,r) \in  \R^3 \times (0,\infty),  S(c,r) \subset \Omega \big\}, \quad \mathcal{S}_m = S(c,r)  $$
Then, for all $k=1\dots N$, for all $t$, $\mathcal{I}_k(t) = S_{m_k(t)}$ with $m_k(t) = \big(c_k(t), r_k(t)\big)$.  It remains to write \eqref{system_2} in this case, more precisely to specify the last two conditions in \eqref{system_2}. Therefore, we introduce the map 
$$ \phi_m(y) := c + r y. $$
For all $k$, for all $t$,   
$\mathcal{S}_{m_k(t)} = \phi_{m_k(t)}(\mathbb{S}^2)$.
Let $t$ fixed, $x \in \mathcal{I}_k(t)$. Let  $y \in \mathbb{S}^2$ such that $\phi_{m_k(t)}(y) = x$.  Then, the map  
$$ \big( M(\tau), X(\tau)\big) = \big( m_k(t+\tau) , \phi_{m_k(\tau)}(y)\big) =  \big( m_k(t+\tau) , c_k(t+\tau) + r_{k}(t+\tau) y\big)  $$
is  is a path in $\bigsqcup_{m \in M} \mathcal{S}_m$ such that $M(0) = m_k(t)$, $M'(0) = m_k'(t)$ and $X(0) = x$. Hence, 
$$ \big\langle V_{m_k(t)}(x) , m'_k(t) \big\rangle = \frac{d}{d\tau} \phi_{m_k(\tau)}(y)\vert_{\tau=t} \cdot n(x) = c'_k(t) \cdot n(x) + r'_k(t) y \cdot n(x) = c'_k(t) \cdot n(x) + r'_k(t) $$ 
using that $n(x)=y$.  Hence, the second-to-last condition in \eqref{system_2} reads 
$$ u_{+}\cdot n =u_{-}\cdot n = c'_k(t) \cdot n(x) + r'_k(t) \text{ on }\mathcal{I}_k\left(  t\right) \quad \forall k,  $$
As regards the last condition, we need to compute 
 $ \big\langle V_{m_k(t)}(x) , \dot{m} \big\rangle $
 for any 
 $$\dot{m} = (\dot{c} , \dot{r})  \in T_{m_k(t)} M = \R^3 \times \R.$$
   We fix $t$ and $x \in \mathcal{S}_{m_k(t)}$. Let $y \in \mathbb{S}^2$ such that $\phi_{m_k(t)}(y) = x$. Using the path 
$$ \big(M(\tau), X(\tau)\big) = \big(M(\tau), \phi_{M(\tau)}(y)\big), \quad \text{ where } \: M(\tau) = \big(c_k(t) + \dot{c} \tau , r_k(t) +  \dot{r} \tau  \big)  $$
satisfying $M(0)= m_k(t)$, $M'(0) = \dot{m}$ and $X(0) = x$ we find 
$$ \big\langle V_{m_k(t)}(x) , (\dot{c} , \dot{r})  \big\rangle = X'(0)  \cdot n(x) = \dot{c} \cdot n(x) + \dot{r} $$
Hence, the last condition in \eqref{system_2} is equivalent to:  for all $\dot{c} \in \R^3, \dot{r} \in \R$, for all $t$ and $k$, 
\begin{align*}
    \int_{\mathcal{I}_k(t)} \big(p_-(\varrho^e_-(t,\cdot))    - p^e_+(t,\cdot) \big)  \left( \dot{c} \cdot n +   \dot{r} \right)   \mathrm{d}\sigma = 0 
    \end{align*}
or equivalently 
\begin{align*}
    \int_{\mathcal{I}_k(t)} \big(p_-(\varrho^e_-(t,\cdot))    - p^e_+(t,\cdot) \big)  n    \mathrm{d}\sigma = 0 \\
    \int_{\mathcal{I}_k(t)} \big(p_-(\varrho^e_-(t,\cdot))    - p^e_+(t,\cdot) \big)   \mathrm{d}\sigma = 0
\end{align*}
Eventually, we gather the boundary conditions at the interfaces with the remaining equations of \eqref{system_2}. We get

\begin{equation}
\left\{
\begin{array}
[c]{l}%
\Omega = \Omega_{+,t} \sqcup \Omega_{-,t} \sqcup \mathcal{I}(t), \quad \Omega_{-,t} \Subset \Omega, \\
\mathcal{I}(t) = \bigcup_{k=1}^N \mathcal{I}_k(t), \quad \mathcal{I}_k(t) =  S(c_k(t),r_k(t))   \quad \forall  k, \\
\varrho_{+}(\partial_{t}u_{+}+u_{+}\cdot\nabla u_{+})+\nabla p_{+}=0\text{ in
}\Omega_{+,t},\\
\partial_{t}\varrho_{-}+\operatorname{div}\left(  \varrho_{-}%
u_{-}\right)  =0\text{ in }  \Omega_{-,t}, \\
\varrho_{-}(\partial_{t}u_{-}+u_{-}\cdot\nabla u_{-})+\nabla
p_{-}\left(  \varrho_{-}\right)  =0 \text{ in }  \Omega_{-,t} \\
u_+ \cdot n = 0 \text{ on } \pa \Omega \\
u_{+}\cdot n=u_{-}\cdot n = c'_k(t) \cdot n + r'_k(t) \text{ on }   \mathcal{I}_{k}\left(  t\right) \quad \forall k,  \\
\int_{\mathcal{I}_{k}\left(  t\right)  }p_{+}n\mathrm{d}\sigma=\int_{\mathcal{I}_{k}\left(  t\right)  }p_{-}\left(  \varrho_{-}\right)  n\mathrm{d}%
\sigma \quad \forall k,\\
\int_{\mathcal{I}_{k}\left(  t\right)  }p_{+}\mathrm{d}\sigma=\int_{\mathcal{I}_{k}\left(  t\right)  }p_{-}\left(  \varrho_{-}\right)  \mathrm{d}%
\sigma \quad \forall k.
\end{array}
\right.  \label{Euler_constraint}
\end{equation}

\begin{remark}
If surface tension is considered then the last equation of $\left(
\text{\ref{Euler_constraint}}\right)  $ will change into%
\[
\int_{\mathcal{I}_{k}\left(  t\right)  }p_{+}\mathrm{d}\sigma=\int_{\mathcal{I}_{k}\left(  t\right)  }p_{-}\left(  \varrho_{-}\right)  \mathrm{d}%
\sigma+(coef)\times8\pi r_{k}\left(  t\right)  .
\]

\end{remark}

\subsubsection{Simplified models}\label{sec_simplified_models}

In this section we present two simplified models that can be obtained from $\left(  \text{\ref{Euler_constraint}}\right)  $ under some additional 
hypothesis or directly as the Euler-Lagrange equations of some  variants of the action functional $\mathcal{A}$. 
A first simplification is to consider that the flow in each bubble
is radial. For each $k=1 \dots N$, we introduce $\varrho_k = \varrho_k(t,r)$ such that 
$\rho_-(t,x) = \varrho_k(t,|x-c_k(t)|)$, and $u_k = u_k(t,r)$ such that $\big( u_-(t,x) - c_k(t) \big) \cdot n(t,x) = u_k(t,|x-c_k(t)|)$,  $x \in B(c_k(t),r_k(t))$.We get 
\begin{equation}
\left\{
\begin{array}
[c]{l}%
\Omega = \Omega_{+,t} \sqcup \Omega_{-,t} \sqcup \mathcal{I}(t), \quad \Omega_{-,t} \Subset \Omega, \\
\mathcal{I}(t) = \bigcup_{k=1}^N \mathcal{I}_k(t), \quad \mathcal{I}_k(t) =  S(c_k(t),r_k(t))   \quad \forall  k, \\
\varrho_{+}(\partial_{t}u_{+}+u_{+}\cdot\nabla u_{+})+\nabla p_{+}=0\text{ in
}\Omega_{+,t},\\
\operatorname{div}u_{+}=0 \text{ in } \Omega_{+,t},\\
\partial_{t}(r^{2}\varrho_{k})+\partial_{r}\left(  r^{2}\varrho_{k}%
u_{k}\right)  =0\text{ in }\left(  0,r_{k}\left(  t\right)  \right) \\
\partial_{t}\left(  r^{2}\varrho_{k}u_{k}\right)  +\partial_{r}\left(
r^{2}\varrho_{k}u_{k}^{2}\right)  +r^{2}\frac{\partial}{\partial r}(p_-\left(
\varrho_{k}\right)  )=0 \text{ in }\left(  0,r_{k}\left(  t\right)  \right) \quad \forall k,\\
u_+ \cdot n = 0 \text{ on } \pa \Omega, \\
u_+(t,\cdot) \cdot n = c'_k(t) \cdot n + r_k'(t) \text{ on } \mathcal{I}_{k}(t)\quad \forall k, \\
u_k(r_k(t)) =   r_k'(t)  \quad \forall k, \\
{\displaystyle\int_{\mathcal{I}_{k}\left(  t\right)  }}
p_{+}\left(  t,x\right)  n\left(  t,x\right)  \mathrm{d}\sigma=0,\quad \forall k,\\
{\displaystyle\int_{\mathcal{I}_{k}\left(  t\right)  }}
p_{+}\left(  t,x\right)  \mathrm{d}\sigma=4\pi r_{k}^{2}\left(  t\right)
p_-\left(  \varrho_{k}\left(  t,r_{k}\left(  t\right)  \right)  \right)  \quad \forall k.
\end{array}
\right.  \label{simplified_eul_1}%
\end{equation}

\bigskip
Another simplification of the system $\left(  \text{\ref{Euler_constraint}}\right)$ is obtained by neglecting the inertia of the bubbles. It follows that for all $k$,  $\na p_-(\rho_k) = 0$, where $\rho_k$ is the density of the $k$-th bubble.  Hence,  $\rho_k = \rho_k(t)$, and by conservation of mass: $\rho_k(t) = \frac{M_k}{\frac{4}{3}\pi r_k(t)^3}$, with $M_k$ the fixed mass of the $k$-th bubble. Eventually, $\left(  \text{\ref{Euler_constraint}}\right)$ reduces to 

\begin{equation}
\left\{
\begin{array}
[c]{l}%
\Omega = \Omega_{+,t} \sqcup \Omega_{-,t} \sqcup \mathcal{I}(t), \quad \Omega_{-,t} \Subset \Omega, \\
\mathcal{I}(t) = \bigcup_{k=1}^N \mathcal{I}_k(t), \quad \mathcal{I}_k(t) =  S(c_k(t),r_k(t))   \quad \forall  k \\
\varrho_{+}(\partial_{t}u_{+}+u_{+}\cdot\nabla u_{+})+\nabla p_{+}=0\text{ in
}\Omega_{+,t},\\
\operatorname{div}u_{+}=0 \text{ in } \Omega_{+,t},\\
u_+ \cdot n = 0 \text{ on } \pa \Omega, \\
u_+(t,\cdot) \cdot n = c'_k(t) \cdot n + r_k'(t) \text{ on } \mathcal{I}_{k}(t)\quad \forall k \\
{\displaystyle\int_{\mathcal{I}_{k}\left(  t\right)  }}
p_{+}\left(  t,x\right)  n\left(  t,x\right)  \mathrm{d}\sigma=0,\quad \forall k\\
{\displaystyle\int_{\mathcal{I}_{k}\left(  t\right)  }}
p_{+}\left(  t,x\right)  \mathrm{d}\sigma=4\pi r_{k}^{2}\left(  t\right)
p_-\left(  \frac{3 M_{k}}{4 \pi r_{k}^{3}}  \right)  \quad \forall k
\end{array}
\right.  \label{eq_sphere_simple}
\end{equation}

\subsection{Ellipsoidal bubbles}\label{sec_ellipsoidal_bubbles}
As another example of our general approach, we want now to allow for ellipsoidal bubbles, enlarging the class of admissible shapes: for all time, we should have  
$$ \mathcal{I}(t) = \bigcup_{k=1}^N \mathcal{I}_k(t), \quad  \mathcal{I}_k(t) = E(c_k(t), S_k(t))$$
where for $c \in \R^3$ and $S \in S_3^{++}(\R)$  the ellipsoid $E(c,S)$ is given by  
$$E(c,S) = \{ x = c + S y, \quad y \in \mathbb{S}^2   \}. $$ 
Again, this framework can be recast in the general one of  Paragraph \ref{sec_formalism}.  We set 
$$ M = \Big\{ m = (c,S) \in  \R^3 \times S_3^{++}(\R),  E(c,r) \subset \Omega \big\},, \quad  \mathcal{S}_m = E(c,S).    $$
 Then, for all $k=1\dots N$, for all $t$, $\mathcal{I}_k(t) = S_{m_k(t)}$ with $m_k(t) = \big(c_k(t), S_k(t)\big)$.  We consider this time  
$$\phi_m(y) = c + Sy. $$
Again, for all $k$, for all $t$,   
$\mathcal{S}_{m_k(t)} = \phi_{m_k(t)}(\mathbb{S}^2)$.  Let  $y \in \mathbb{S}^2$ such that $\phi_{m_k(t)}(y) = x$.  Proceeding exactly as in the previous paragraph, we find that the second-to-last condition of \eqref{system_2} involves 
$$ \big\langle V_{m_k(t)}(x) , m'_k(t) \big\rangle = c'_k(t) \cdot n(x) + S'_k(t) y \cdot n(x) = c'_k(t) \cdot n(x) + S'_k(t) S_k(t)^{-1}\big(x-c_k(t)\big) \cdot n(x)   $$ 
Hence, the second-to-last condition in \eqref{system_2} reads 
$$ u_{+}\cdot n=u_{-}\cdot n = c'_k(t) \cdot n(x) +  S'_k(t) S_k(t)^{-1}\big(x-c_k(t)\big) \cdot n(x)  \text{ on }\mathcal{I}_k\left(  t\right) \quad \forall k,  $$
As regards the last condition, we need to compute 
 $ \big\langle V_{m_k(t)}(x) , \dot{m} \big\rangle $
 for any 
$$\dot{m} = (\dot{c}, \dot{S})  \in T_{m_k(t)} M = \R^3 \times S_3(\R).$$
We consider the path in $M$
$$ M(\tau) = \big( c_k(t) + \dot{c} \tau, S(\tau) \big), \quad S(\tau) \in S_3^{++}(\R), \quad S(0) = S_k(t), \quad S'(0) = \dot{S},$$
as well as the corresponding path in  $\bigsqcup_{m \in M} \mathcal{S}_m$ given by $\tau \rightarrow \big(M(\tau),\phi_{M(\tau)}(y)\big)$. 
 We end up with 
$$ \big\langle V_{m_k(t)}(x) , (\dot{c} , \dot{S})  \big\rangle =  \dot{c} \cdot n(x) + \dot{S} S^{-1} (x - c_k(t)) \cdot n(x)  $$
Hence, the last condition of \eqref{system_2} reads: for all $\dot{c} \in \R^3$, $\dot{S} \in \mathcal{S}_3(\R)$, 
\begin{align*}
    \int_{\mathcal{I}_k(t)} \big(p_-(\varrho^e_-(t,\cdot))    - p^e_+(t,\cdot) \big)  \left( \dot{c} \cdot n  + \dot{S} S^{-1} (x - c_k(t)) \cdot n(x) \right)    \mathrm{d}\sigma(x) = 0 
    \end{align*}
This is equivalent to 
\begin{align*}
   & \int_{\mathcal{I}_k(t)} \big(p_-(\varrho^e_-(t,\cdot))    - p^e_+(t,\cdot) \big)  n    \mathrm{d}\sigma = 0 \\
    & \int_{\mathcal{I}_k(t)} \big(p_-(\varrho^e_-(t,x))    - p^e_+(t,x) \big) \left[ n(x) \otimes \big( S_k(t)^{-1} (x-c_k(t)) \big) \right]^{sym}     \mathrm{d}\sigma(x) = 0
\end{align*}
with notation $[M]^{sym} = \frac{1}{2}(M+M^t)$.
We end up with the reduced system 
\begin{equation}
\left\{
\begin{array}
[c]{l}%
\Omega = \Omega_{+,t} \sqcup \Omega_{-,t} \sqcup \mathcal{I}(t), \quad \Omega_{-,t} \Subset \Omega, \\
\mathcal{I}(t) = \bigcup_{k=1}^N \mathcal{I}_k(t), \quad \mathcal{I}_k(t) =  E(c_k(t),r_k(t))   \quad \forall  k, \\
\varrho_{+}(\partial_{t}u_{+}+u_{+}\cdot\nabla u_{+})+\nabla p_{+}=0\text{ in
}\Omega_{+,t},\\
\partial_{t}\varrho_{-}+\operatorname{div}\left(  \varrho_{-}%
u_{-}\right)  =0\text{ in }  \Omega_{-,t}, \\
\varrho_{-}(\partial_{t}u_{-}+u_{-}\cdot\nabla u_{-})+\nabla
p_{-}\left(  \varrho_{-}\right)  =0 \text{ in }  \Omega_{-,t} \\
u_+ \cdot n = 0 \text{ on } \pa \Omega \\
u_{+}\cdot n=u_{-}\cdot n = c'_k(t) \cdot n + S'_k(t) S_k(t)^{-1}(\cdot - c_k(t)) \cdot n \text{ on }   \mathcal{I}_{k}\left(  t\right) \quad \forall k,  \\
\int_{\mathcal{I}_{k}\left(  t\right)  }p_{+}n\mathrm{d}\sigma=\int_{\mathcal{I}_{k}\left(  t\right)  }p_{-}\left(  \varrho_{-}\right)  n\mathrm{d}%
\sigma \quad \forall k,\\
\int_{\mathcal{I}_{k}\left(  t\right)  }(p_{+} - p_{-}\left(  \varrho_{-}\right)) \left[ n \otimes \big( S_k(t)^{-1} (\cdot-c_k(t)) \big) \right]^{sym}    \mathrm{d}%
\sigma \quad \forall k.
\end{array}
\right.  \label{Euler_constraint_ellipse}%
\end{equation}

\section{Well-posedness of the reduced models} \label{sec_WP_reduced}
In the previous sections, we discussed the derivation of the standard equations for incompressible/compressible two-phase flows (Section \ref{section_immiscible}) and the derivation of reduced models with shape constraints (the whole Section  \ref{sec_shape_constrained}). Namely, we derived the equations that a critical point of some action functional should satisfy. Whether or not such critical point exists, or a solution to our models exists, is a different question, that we will only touch upon here.  

As pointed out in Remark \ref{rem_compatibility condition}, in  all the models seen before, if the domain $\Omega$ is bounded, the incompressibility condition $\operatorname{div} u_+ = 0$ requires compatibility conditions on the initial data, to be preserved through time. Indeed, taking into account the boundary condition $u \cdot n = 0$ at $\pa \Omega$, it is necessary that  
$$ \int_{\mathcal{I}(t)} u \cdot n = 0 \quad \forall t \ge 0.$$
This constraint is particularly stringent in the reduced models, for instance for the model \eqref{Euler_constraint}. For instance, for a single spherical bubble in a container ($N=1$, $\Omega$ bounded), the radius of the ball has to be constant ! A natural way to relax this constraint is to take $\Omega = \R^3$. For instance,  there exists a divergence-free field in $\R^3\setminus B(0,1)$ such that $u \cdot n  = 1$ at $\pa B(0,1)$: take $u = \frac{x}{|x|^2}$. 

\medskip
As emphasized in the introduction of the present paper, the well-posedness of the complete two-phase flow model  \eqref{system_1}  has been studied from a rigorous perspective in several articles. The analysis of reduced models for bubbly flows with shape constraints is however much more limited: as far as we know, dedicated articles  only deal with spherical bubbles, see the references in the introduction. Notably, the question of well-posedness of System \eqref{system_2} is open. 

\medskip
We shall make partial progress on this question here, by showing the well-posedness of a simplified version of \eqref{system_2}. For $m \in M$, we recall that   $\Omega_m$ is the smooth domain bounded by $\mathcal{S}_m$, see Section  \ref{sec_formalism}. We make three main hypotheses: 
\begin{enumerate}
\item $\Omega = \R^3$, that is the liquid fills in the whole space. 
\item The liquid flow is supposed to be irrotational : 
$\textrm{curl } u_+ = 0$. 
\item The pressure $p_k$, or equivalently the density $\rho_k$, is homogeneous in space. By conservation of mass, we deduce 
$$ p_k(t) = p_-(\rho_k(t)) = p_-\left(\frac{M_k}{|\Omega_{m_k}(t)|} \right)$$
where $M_k$ is the fixed mass of the $k$-th bubble, $k=1\dots N$. 
\end{enumerate}
The first hypothesis is not crucial, and one could consider a bounded domain $\Omega$, see the discussion of Paragraph \ref{sec_bounded_cavity}. 
Hence, we consider the following model:
\begin{equation} 
\left\{
\begin{array}
[c]{l}%
\partial_{t}u + u\cdot\nabla u +\nabla p =0\text{ in
} \R^3 \setminus \overline{\cup_{k=1}^N \Omega_{m_k(t)}},\\
\textrm{curl } u = 0, \quad \operatorname{div} u =0\text{ in
} \R^3 \setminus \overline{\cup_{k=1}^N \Omega_{m_k(t)}},\\
u(t,\cdot)  \cdot n = \langle V_{m_k(t)} , m'_k(t) \rangle \text{ on } \pa \Omega_{m_k(t)}, \quad \forall k,  \\
  \int_{\pa \Omega_{m_k(t)}} \big(p_-(\varrho_-(t,\cdot))   - p_k(t) \big) V_{m_k(t)}(\cdot)  \mathrm{d}\sigma  = 0,\quad \forall  k.
\end{array}
\right.  \label{system_3}%
\end{equation}
Compared to \eqref{system_2},  we suppressed the lowerscript $+$ for the fluid quantities and normalized the fluid density to $1$.  We recall that the map $(m,x) \rightarrow V_m(x)$ is defined in \eqref{defi_normal_velocity}. One may additionally prescribe a non-zero constant pressure at infinity $p \xrightarrow[|x| \rightarrow +\infty]{} p_\infty$.  
We will show that for an appropriate class of initial conditions, this system can be further reduced to a system of second order ODE's on $(m_1, \dots , m_N)$ which will have a unique local in time solution.  

\subsection{Shape-dependent gradient fields} \label{sec_shape_dependent}
To reformulate \eqref{system_3} as a set of ODEs, we will rely on some gradient fields, that we will call shape-dependent. We introduce the parameter set
$$ \bfM = \Big\{ \bfm = (m_1, \dots, m_N) \in M^N, \: \overline{\Omega_{m_i}} \cap \overline{\Omega_{m_j}} = \emptyset \quad \forall i \neq j \Big\} $$   
which is an open subset of $M^N$. Let $T\bfM$ its tangent bundle. For any
$$ (\bfm , \dot{\bfm}) \in T\bfM, \quad  \text{with } \: \dot{\bfm} = \left( \dot{m}_1, \dots, \dot{m}_N \right) $$
we introduce $\phi_{\bfm,\dot{\bfm}}$ the solution of 
\begin{equation} \label{system_phi_m}
\left\{
    \begin{aligned}
        -\Delta \phi_{\bfm,\dot{\bfm}} & = 0, \quad \text{ in } \R^3 \setminus \overline{\cup_{k=1}^N \Omega_{m_k}} \\
        \pa_n  \phi_{\bfm,\dot{\bfm}}  & = \big\langle V_{m_k}(\cdot) , \dot{m}_k \big\rangle, \quad \text{ on } \pa \Omega_{m_k} \quad \forall k. 
    \end{aligned}
    \right.
\end{equation}
We recall that the normal vector $n$ points outward $\Omega_{m_k}$. We collect some properties of $\phi_{\bfm,\dot{\bfm}}$ in the next two lemmas. 
\begin{lemma} \label{lemma_phi_m}
For any $(\bfm , \dot{\bfm}) \in T\bfM$, there exists a unique variational solution $\phi_{\bfm,\dot{\bfm}}$ of \eqref{system_phi_m} such that 
$$\na \phi_{\bfm,\dot{\bfm}} \in L^2(\R^3 \setminus \overline{\cup_{i=1}^N \Omega_{m_i}}), \quad \phi_{\bfm,\dot{\bfm}} \in  L^6(\R^3 \setminus \overline{\cup_{i=1}^N \Omega_{m_i}}).$$
Moreover, if the following non-degeneracy condition is satisfied: 
\begin{equation} \label{nondegeneracy}
\forall (m,\dot{m}) \in TM, \quad \dot{m} \neq 0, \quad  \exists x \in \mathcal{S}_m =  \pa \Omega_m, \: \big\langle V_m(x) , \dot{m} \big\rangle \neq 0  
\end{equation}
then, for any $\bfm \in \bfM$ and any basis $(\mathbf{e}_1, \dots \mathbf{e}_p)$  of $T_{\bfm} \bfM$,  with $p = \dim \bfM = N \dim M$, the family 
$\big(\phi_{\bfm,\mathbf{e}_1}$, \dots, $\phi_{\bfm,\mathbf{e}_p}\big)$  is free.  
\end{lemma}
\begin{proof}
To lighten notations, we write $\phi$ instead of $\phi_{\bfm,\dot{\bfm}}$, $\: \Omega_k$ instead of $\Omega_{m_k}$ and $\:a_k = \big\langle V_{m_k}(\cdot) , \dot{m}_k \big\rangle$. Hence, we need to solve the problem 
\begin{equation} \label{system_phi}
\left\{
\begin{aligned}
\Delta \phi & = 0 \quad \text{ in } \: \R^3 \setminus \overline{\cup_{k} \Omega_k}, \\
\pa_n \phi & = a_k \quad \text{ on } \: \pa \Omega_k, \quad \forall k
\end{aligned}
\right.
\end{equation}
Let $(x_1, r_1)$, \dots, $(x_N, r_N)$ such that for all $k=1,\dots,N$, $\overline{B(x_k,r_k)} \subset \Omega_k$. We define for all $k$ and all $x \neq x_k$,  $\phi_k(x) = -\frac{r_k^2}{4\pi|x-x_k|}$. It is easily checked that $\phi_k$ is a harmonic function in $\R^3\setminus \{x_k\}$, with 
$$ \int_{\pa \Omega_k} \pa_n \phi_k = \int_{\pa B(x_k,r_k)} \pa_n \phi_k = 1, \quad \text{ while } \: \int_{\pa \Omega_j} \pa_n \phi_k = 0 \quad \forall j \neq k.   $$
We then set $\psi = \phi - \sum_{k=1}^N \big( \int_{\pa \Omega_k} a_k  \big) \phi_k$, so that the system on $\phi$ is equivalent to 
\begin{equation} \label{system_psi}
\left\{
\begin{aligned}
\Delta \psi & = 0 \quad \text{ in } \: \R^3 \setminus \overline{\cup_{k} \Omega_k}, \\
\pa_n \psi & = a'_k = a_k - \sum_{j=1}^N \big( \int_{\pa \Omega_j} a_j  \big) \pa_n \phi_j \quad \text{ on } \: \pa \Omega_k, \quad \forall k
\end{aligned}
\right.
\end{equation}
Let $B$ an open ball containing $\cup_k \overline{\Omega_k}$ and $\Omega' = B \setminus \overline{\cup_k \Omega_k}$. Let 
$$ V = \Big\{ \varphi \in L^2(\Omega') \cap L^2_{loc}(\R^3 \setminus \overline{\cup_k \Omega_k}), \na \varphi \in L^2(\R^3 \setminus \overline{\cup_k \Omega_k}), \quad \int_{\Omega'} \varphi = 0 \Big\}$$
Thanks to the Poincaré inequality: $\|f\|_{L^2(\Omega')} \le C \|\na f\|_{L^2(\Omega')}$ for all $f$ s.t. $\int_{\Omega'} f =0$, $V$ is a Hilbert space equipped with 
$\| \varphi\| = \|\na \varphi\|_{L^2(\R^3 \setminus \overline{\cup_k \Omega_k})}$. Moreover, combining this Poincaré inequality and the trace inequality  $\|f\|_{H^{1/2}(\pa \Omega_k)} \le C \|f\|_{H^1(\Omega')}$, it is easily seen that the variational formulation: for all $\varphi \in V$ 
$$ \int_{\R^3 \setminus \overline{\cup \Omega_k}} \na \psi \cdot \na \varphi = - \sum_k \int_{\pa \Omega_k} a'_k \varphi,   $$
has a unique solution $\psi \in V$. But as $\int_{\pa \Omega_k} a'_k = 0$ for all $k$, this equality remains true replacing $\varphi$ by $\varphi - C$ for any constant $C$. In other words, the equality remains true for all $\varphi$ that do not satisfy the zero mean condition $\int_{\Omega'} \varphi = 0$. Hence, the solution $\psi \in V$ satisfies \eqref{system_psi}. Moreover, it is well-known that there exists a constant $\psi_\infty$ such that $\psi - \psi_\infty \in L^6(\R^3\setminus \cup_k \Omega_k)$, see for instance \cite[Theorem II.6.1]{Galdi}. Back to the original problem, $\phi = \psi - \psi_\infty + \sum_{k=1}^N \big( \int_{\pa \Omega_k} a_k \cdot n \big) \phi_k$ is the unique solution of \eqref{system_phi} that satisfies the requirements of the lemma. 

\medskip
For the second part of the lemma, we introduce real numbers $\lambda_1, \dots, \lambda_p$ such that 
$$\sum_{i=1}^p \lambda_i \phi_{\bfm,\mathbf{e}_i} = 0 $$
This means that $\phi_{\bfm, \dot{\bfm}} = 0$ where $\dot{\bfm} = \sum_{i=1}^p \lambda_i \mathbf{e}_i$. This implies that the corresponding Neumann data are zero : for all $k$, for all $x \in \pa \Omega_{m_k}$, $\big\langle V_{m_k}(x) , \dot{m}_k \big\rangle = 0$. By the non-degeneracy condition \eqref{nondegeneracy}, this implies that $\dot{\bfm} = 0$, hence that all $\lambda_i$ are zero.     
\end{proof}

\begin{lemma} \label{lemma_regularity_phi_m}
Let $\mathbf{E}$ a smooth vector field on $\bfM$.  The map  
$$ \bigcup_{\bfm \in \bfM}  \{ \bfm\} \times \Big( \R^3 \setminus \overline{\cup_{k=1}^n \Omega_{m_k}} \Big) \rightarrow \R, \quad  (\bfm,x) \rightarrow \phi_{\bfm, \mathbf{E}(\bfm)}(x) $$
is smooth. Moreover, for any smooth map $\tau \rightarrow \bfm(\tau)$ in $\bfM$, one has for all  $\alpha \in  \mathbb{N}^3$ with $|\alpha| \ge 1$:  
$$ \pa_{\tau} \pa^\alpha_x \phi_{\bfm(\tau), \mathbf{E}(\bfm(\tau))}  \in L^2\left(\R^3 \setminus \overline{\cup_k \Omega_{m_k}}\right).$$
\end{lemma}
\begin{proof}
 The key difficulty is that the domain of definition of $\phi_{\bfm,\mathbf{E}(\bfm)}$ depends on $\bfm$. This problem of regularity with respect to the domain was analyzed in details in \cite{Henrot}, especially for solutions of a  Neumann problem in a bounded domain:  see \cite[chapter 5.5]{Henrot}. In our system \eqref{system_phi_m}, we have  an extra dependence on $\bfm$ in the  boundary data, and the domain is not bounded but of the form $\R^3 \setminus \cup_k \Omega_{m_k}$. Still,  adapting the approach of \cite{Henrot} to these changes is straightforward. Therefore,  we just remind here the big steps of this approach. Let $\tau \rightarrow \bfm(\tau)$ a smooth path in $\bfM$ defined for $\tau$ in a neighborhood of zero.  First, one builds over the fluid domain $\R^3 \setminus \cup_k \Omega_{m_k(\tau)}$ some velocity field  $v(\tau, \cdot)$ such that $v(\tau, \cdot)\vert_{\pa \Omega_{m_k(\tau)}} = \big\langle V_{m_k(\tau)} , m'_k(\tau) \big\rangle n$. In other words, we extend to the whole fluid domain  the (normal) velocity of the interface $\cup_{k} \pa \Omega_{m_k(\tau)}$. Denoting $\Theta_\tau : \R^3 \setminus \cup_k \Omega_{m_k(0)} \rightarrow \R^3 \setminus \cup_k \Omega_{m_k(\tau)}$ the associated flow map, one can through the change of variable $x = \Theta_\tau(y)$ turn the Neumann problem satisfied by 
 $\phi_{\bfm(\tau), \mathbf{E}(\bfm(\tau))}$ into a similar problem on a fixed domain, where the Laplacian is replaced by an elliptic operator with variable coefficients depending smoothly on $\tau$.  Smoothness with respect to $y$ follows from ellipticity, while smoothness with respect to $\tau$ can then be established through an implicit function theorem. See \cite[Theorem 5.5.1]{Henrot} for details.  Back to the original moving domain, this provides smoothness of the map $(\tau,x) \rightarrow \phi_{\bfm(\tau), \mathbf{E}(\bfm(\tau))}(x)$ for $x$ away from the boundary. Moreover, differentiating the elliptic equations in the fixed domain with respect to  $\tau$, and translating into the original moving domain, it shows that $\pa_\tau \phi_{\bfm(\tau), \mathbf{E}(\bfm(\tau))}$ satisfies itself a new Neumann problem, with data involving   $\phi_{\bfm(\tau), \mathbf{E}(\bfm(\tau))}$. See \cite[Theorem 5.5.2]{Henrot}. In particular, we find that for all $\alpha \in  \mathbb{N}^3$ with $|\alpha| \ge 1$:  
 $$ \pa_\tau \pa^\alpha_x \phi_{\bfm(\tau), \mathbf{E}(\bfm(\tau))} \in L^2(\R^3 \setminus \overline{\cup_k \Omega_{m_k(\tau)}}) $$
 (the case $|\alpha|=1$ is directly given by the variational formulation of the new Neumann problem, the case $|\alpha| \ge 1$ is deduced from  elliptic regularity). As the path $\tau \rightarrow \bfm(\tau)$ is arbitrary, we can take local coordinates $\tau_1, \dots, \tau_p \in \R^p$ parametrizing $\bfM$ and consider the maps $\tau_i \rightarrow \bfm(\tau_1, \dots, \tau_i, \dots,\tau_p)$. We recover in this way all partial derivatives of first order. 
 Reiterating the reasoning with the partial derivatives (which as we saw solve similar Neumann problem), we obtain regularity in $\bfm$ at any order.  
\end{proof}

\subsection{ODE system for curl-free flows} \label{sec_ODEs}
Let $\bfm_0  \in \bfM$. For $V$ a small enough neighborhood of  $\bfm_0$, there exist  smooth vector fields $\mathbf{E}_1$, \dots, $\mathbf{E}_p$ and smooth one-forms $\mathbf{E}_1^*, \dots, \mathbf{E}_p^*$ defined on $V$ such that: for each $\bfm \in V$,  
\begin{equation} \label{basis_dual_basis}
\big(\mathbf{E}_1(\bfm), \dots, \mathbf{E}_p(\bfm)\big) \: \text{ is a basis of } \: T_{\bfm} \bfM, \quad  \text{ and for all $i,j$}, \: \big\langle \mathbf{E}_i^*(\bfm) , \mathbf{E}_j(\bfm) \rangle = \delta_{ij} 
\end{equation}
Any vector $\dot{\bfm} \in T_{\bfm}\bfM$ is decomposed in a unique way in this basis: 
$$\dot{\bfm} = \sum_{i=1}^p \big\langle \mathbf{E}_i^*(\bfm) , \dot{\bfm} \rangle  \mathbf{E}_i(\bfm).$$ 
Let now $\bfm = \bfm(t) = \big(m_1(t), \dots, m_N(t)\big)$ and $u = u(t,x)$  satisfying \eqref{system_3} with $\bfm(0) = \bfm_0$. As  $u$ is curl-free it can be written as $u(t,x) = \nabla_x \phi(t,x)$ and as it is divergence-free 
$$\Delta_x \phi(t,\cdot) = 0, \quad \text{in } \:   \R^3 \setminus \overline{\cup_{k=1}^N \Omega_{m_k(t)}}.$$ 
Moreover, the boundary condition on the normal velocity reads 
$$ \pa_n \phi(t, \cdot) = \big\langle V_{m_k(t)}(\cdot) , m'_k(t) \big\rangle, \quad \text{ on } \pa \Omega_{m_k(t)} \quad \forall k, \forall t.  $$
It follows, with the notations of Paragraph \ref{sec_shape_dependent}, that
$$ \forall t, \quad u(t, \cdot) = \nabla \phi_{\bfm(t), \bfm'(t)}. $$
For $t$ small enough, by continuity, $\bfm(t) \in V$. By Lemma \ref{lemma_phi_m}, under the non-degeneracy condition \eqref{nondegeneracy}, one can decompose in a unique way $\phi_{\bfm(t), \bfm'(t)}$ as 
 $$ \phi_{\bfm(t), \bfm'(t)} = \sum_{i=1}^p \big\langle \mathbf{E}_i^*(\bfm(t)) , \bfm'(t) \big\rangle  \, \phi^i_{\bfm(t)}, \quad \text{where for all $\bfm$ in $V$, } \phi^i_{\bfm} = \phi_{\bfm,\mathbf{E}_i(\bfm)}.$$
Finally, for all $t$, 
\begin{equation} \label{decomposition_u}
u(t,\cdot) =  \sum_{i=1}^p \big\langle \mathbf{E}_i^*(\bfm(t)) , \bfm'(t) \big\rangle  \, \na \phi^i_{\bfm(t)}
\end{equation}
and the only unknown left is the function $t \rightarrow \bfm(t)$ encoding the dynamics of the bubbles. To obtain this dynamics, 
we will perform computations that have strong similarities with those leading to the added mass effect for bidimensional irrotational inviscid flows, such as described in \cite{HouotMunnier,GlassSueurTakahashi2012,GlassLacaveMunnier}. Let us come back to the last boundary condition of \eqref{system_3}, where  we recall that $p_k = p_k(t)$ is the homogeneous pressure in the $k$-th bubble. Note that it depends on $t$ only through the volume of $\Omega_{m_k(t)}$, in particular only through $m_k(t)$: $p_k(t) = P_k(m_k(t))$.  For each $\bfm = (m_1, \dots, m_k) \in \bfM$, for all $x \in \cup_k \pa \Omega_{m_k}$, we define the linear form  $V_{\bfm}(x)$ by:  
$$ \forall \: \dot{\bfm} = \big( \dot{m}_1, \dots , \dot{m}_N \big), \quad    \big\langle V_{\bfm}(x) , \dot{\bfm} \big\rangle = \big\langle V_{m_k}(x) , \dot{m}_k \big\rangle \quad \text{ if } x \in \pa \Omega_{m_k}.  $$
With this notation, it is easily seen that the last equality in \eqref{system_3} can be written 
$$ \int_{\cup_k \pa \Omega_{m_k(t)}} \Big( p(t,\cdot) - \sum_{k} 1_{\pa \Omega_{m_k(t)}} p_k(t) \Big) V_{\bfm(t)}(\cdot) d\sigma = 0   $$
or for $t$ small enough, for all $i=1,\dots,p$, 
$$ \int_{\cup_k \pa \Omega_{m_k(t)}} \Big( p(t,\cdot) - \sum_{k} 1_{\pa \Omega_{m_k(t)}} p_k(t) \Big) \big\langle V_{\bfm(t)}(\cdot) , \mathbf{E}_i(\bfm(t)) \big\rangle d\sigma = 0.   $$
We rewrite it as: for all $i=1,\dots,p$, 
$$ \int_{\cup_k \pa \Omega_{m_k(t)}} \Big( p(t,\cdot) - p_\infty \Big)  \big\langle V_{\bfm(t)}(\cdot) , \mathbf{E}_i(\bfm(t)) \big\rangle d\sigma = F_i(\bfm(t))   $$
where 
$$F_i(\bfm) =  \int_{\cup_k \pa \Omega_{m_k}}  \left( \sum_{k} 1_{\pa \Omega_{m_k}} P_k(m_k) - p_\infty  \right)  \big\langle V_{\bfm}(\cdot) , E_i(\bfm) \big\rangle d\sigma. $$
We can then reformulate  the left-hand side: 
\begin{align*}
& \int_{\cup_k \pa \Omega_{m_k(t)}} \big( p(t,\cdot) - p_\infty \big) \,  \big\langle V_{\bfm(t)}(\cdot) , \mathbf{E}_i(\bfm(t)) \big\rangle d\sigma \\
= & \int_{\cup_k \pa \Omega_{m_k(t)}} \big( p(t,\cdot) - p_\infty \big) \, \pa_n \phi^i_{\bfm(t)}  d\sigma  =  - \int_{\R^3 \setminus \overline{\cup_k \Omega_{m_k(t)}}}  \na p \cdot \na \phi^i_{\bfm(t)} \\
= &  \int_{\R^3 \setminus \overline{\cup_k \Omega_{m_k(t)}}} \left( \pa_t u + u \cdot \na u \right) \cdot \na \phi^i_{\bfm(t)} \\
\end{align*}
The idea is  then to replace $u$ by its expression in \eqref{decomposition_u}. Differentiation with respect to time and space is possible, thanks to the regularity of the vector fields $\mathbf{E}_i$ and the dual one-forms $\mathbf{E}_i^*$, and thanks to the regularity results of Lemma \ref{lemma_regularity_phi_m}. Denoting $\dot{\bfm}_j(t) = \big\langle \mathbf{E}_j^*(\bfm(t)) , \bfm'(t) \big\rangle \in \R$, so that 
$$ u(t,\cdot) = \sum_{j=1}^p \dot{\bfm}_j(t) \na \phi^j_{\bfm(t)} $$
we end up with a system of the form:  for all $i=1,\dots,p$,
\begin{align*}
    \sum_{j=1}^p  \dot{\bfm}_j'(t) \int_{\R^3 \setminus \cup_k \Omega_{m_k(t)}} \na \phi^j_{\bfm(t)} \cdot \na \phi^i_{\bfm(t)} =   G_i(\bfm(t), \bfm'(t)) 
\end{align*}
for some smooth $G_i$. The Gram matrix $\left( \int_{\R^3 \setminus \cup_k \Omega_{m_k}} \na \phi^j_{\bfm} \cdot \na \phi^i_{\bfm} \right)_{1 \le i,j \le p}$ is invertible because $\big(\phi^i_{\bfm}\big)_{1 \le i \le p}$ is free, and so we get: for all $i=1,\dots,p$ 
$$  \dot{\bfm}_i'(t) = \tilde{G}_i(\bfm(t),\bfm'(t)) $$
for some smooth $\tilde{G}_i$. But differentiating the relation $\bfm'(t) = \sum_i \dot{\bfm}_i(t) \mathbf{E}_i(\bfm(t))$, we find that 
$$ \bfm''(t) = \sum_i \dot{\bfm}'_i(t) \mathbf{E}_i(\bfm(t)) +  \na_{\bfm'(t)} \mathbf{E}_i(\bfm(t))   $$
where for two vector fields $\mathbf{X},\mathbf{Y}$ 
 on $\bfM$,  $\na_{\mathbf{X}} \mathbf{Y}$ is the standard covariant derivative inherited from the one of $\R^{dN}$ (remember that $M$ is a submanifold of $\R^d$, so that $\bfM$ is a submanifold of $\R^{dN}$). 
So, the system can be reformulated as the second order system
$$ \bfm''(t) = \mathcal{G}(\bfm(t),\bfm'(t)) $$
for some smooth $\mathcal{G}$ on $T \bfM$. Starting from any $(\bfm_0, \dot{\bfm}_0) \in T \bfM$, it has a local in time smooth solution $t \rightarrow \bfm(t)$, which in turn provides a local in time smooth solution of \eqref{system_3}, defining $u$ through \eqref{decomposition_u}. 

\subsection{Example of one spherical bubble}
In this paragraph, we consider the important example of one spherical bubble in a curl-free flow: 
\begin{align*}
N = 1, \quad m(t) = \big(c(t), r(t)\big),  \quad \Omega_{m(t)} = B(t) = B(c(t), r(t)), \quad \big\langle V_{m}, \dot{m} \big\rangle = \dot{c} \cdot n + \dot{r}.
\end{align*}
In such a case, the derivation of the ODE system seen in the previous paragraph can be slightly shortened. First, in view of \eqref{eq_sphere_simple},  system \eqref{system_3} reads:
\begin{equation} 
\left\{
\begin{array}
[c]{l}%
\partial_{t}u + u\cdot\nabla u +\nabla p =0\text{ in
} \R^3 \setminus B(t),\\
\textrm{curl } u = 0, \quad \operatorname{div} u =0\text{ in
} \R^3 \setminus B(t),\\
u(t,\cdot)  \cdot n = c'(t) \cdot n + r'(t) \text{ on } \pa B(t),  \\
\displaystyle\int_{ \pa B(t)}
p\left(  t,x\right)  n\left(  t,x\right)  \mathrm{d}\sigma=0\\
\displaystyle\int_{ \pa B(t)}
p\left(t,x\right)  \mathrm{d}\sigma = 4\pi r(t)^{2}
p_B\left(t \right) 
\end{array}
\right.  \label{eq_sphere_curl_free}%
\end{equation}
where $p_B(t) = p_-\left( \frac{3M}{4\pi r(t)^3} \right)$ is the pressure of the bubble.
We introduce the potential $\phi = \phi(t,x)$ such that $u = \na_x \phi$, solving 
$$ \Delta_x \phi(t, \cdot) = 0 \quad \text{ in } \: \R^3\setminus B(t), \quad \pa_n \phi\vert_{\pa B(t)} = c'(t) \cdot n + r'(t) $$
The solution is  given by 
 $$\phi(t,x) = - \frac{r(t)^2 r'(t)}{|x-c(t)|} - \frac{r(t)^3}{2 |x-c(t)|^3} c'(t) \cdot (x - c(t)).  $$
To determine the dynamics of the center of mass $c(t)$ and radius $r(t)$ of the bubble, it is convenient to write directly  the Bernoulli equation 
$$ \pa_t \phi + \frac{1}{2}|\na \phi|^2 + p - p_\infty = 0$$
and to evaluate it at the boundary of the bubble. Straightforward but tedious computations yield: 
\begin{align*} 
\pa_t \phi\vert_{\pa B(t)}  & = - r(t) r''(t) -  2r'(t)^2 - \frac52 r'(t) (c'(t) \cdot n) - \frac32 (c'(t) \cdot n)^2 - \frac12 r(t) c''(t) \cdot n + \frac{1}{2} |c'(t)|^2 \\
\na \phi\vert_{\pa B(t)}  & = r'(t) n + \frac32 (c'(t) \cdot n) n - \frac12 c'(t), \\ 
\frac12 |\na \phi\vert_{\pa B(t)}|^2  &= \frac12 \left( r'(t)^2 + \frac34 (c'(t) \cdot n)^2 + \frac{1}{4} |c'(t)|^2 + 2  (c'(t) \cdot n) r'(t) \right) 
\end{align*}
so that by Bernoulli equation we get
$$ - r(t) r''(t) - \frac32 r'(t)^2 - \frac{3}{2} r'(t) (c'(t)\cdot n) - \frac98 (c'(t) \cdot n)^2 - \frac12 r(t) c''(t) \cdot n + \frac58 |c'(t)|^2 + p\vert_{\pa B(t)} - p_\infty = 0 $$
Integrating over $\pa B(t)$, using the last relation in \eqref{eq_sphere_curl_free}, one finds (remember $\dashint_{\pa B} n \cdot n = \frac{1}{3}$): 
$$ - r(t) r''(t) - \frac32 r'(t)^2 + \frac14 |c'(t)|^2 + p_B(t) - p_\infty = 0 $$
Moreover, multiplying by the normal vector $n$, integrating over $\pa B(t)$ and using that $\int_{\pa B(t)} p(t,\cdot) n = 0$, we get 
$$ \frac32 \frac{r'(t)}{r(t)} c'(t) + c''(t) = 0. $$
Those two equations are the second order ODE's describing the dynamics of the bubble, that could be found previously in \cite{hermans1973}. If we neglect the translation of the bubble ($c=0$), we recover the famous {\em Rayleigh-Plesset equation}: 
$$ - r(t) r''(t) - \frac32 r'(t)^2 + p_B(t) - p_\infty = 0 $$

\subsection{Extension to a bounded cavity} \label{sec_bounded_cavity}
The well-posedness of \eqref{system_3} established in Paragraphs \ref{sec_shape_dependent}-\ref{sec_ODEs} was restricted to  fluid domains of the form $\R^3 \setminus \overline{\cup_k \Omega_{m_k(t)}}$, that is to fluids  filling the whole space. A natural problem is to extend this well-posedness result to the case of a bounded cavity $\Omega$. As discussed in Remark \ref{rem_compatibility condition} or at the beginning of Section \ref{sec_WP_reduced}, there should be  in this case extra compatibility conditions, due to volume constraints induced by incompressibility of the fluid. Indeed, coming back to the approach developped in  Paragraphs \ref{sec_shape_dependent}-\ref{sec_ODEs}, one issue stems from system \eqref{system_phi_m}, that should now be replaced by 
\begin{equation} \label{system_phi_m_bis}
\left\{
    \begin{aligned}
        -\Delta \phi_{\bfm,\dot{\bfm}} & = 0, \quad \text{ in } \Omega \setminus \overline{\cup_{k=1}^N \Omega_{m_k}} \\
        \pa_n  \phi_{\bfm,\dot{\bfm}}  & = \big\langle V_{m_k}(\cdot) , \dot{m}_k \big\rangle, \quad \text{ on } \pa \Omega_{m_k} \quad \forall k, \\
        \pa_n  \phi_{\bfm,\dot{\bfm}}  & = 0, \quad \text{ on } \pa \Omega.
    \end{aligned}
    \right.
\end{equation}
But integrating the first equation over the fluid domain, we find 
\begin{align*}
    0 = \int_{\Omega \setminus \overline{\cup_{k=1}^N \Omega_{m_k}}} \Delta \phi_{\bfm,\dot{\bfm}}  & = \int_{\pa \Omega} \pa_n \phi_{\bfm,\dot{\bfm}} - \sum_k \int_{\pa \Omega_{m_k}} \pa_n \phi_{\bfm,\dot{\bfm}} 
    = - \sum_k \int_{\pa \Omega_{m_k}} \big\langle V_{m_k} , \dot{m}_k \big\rangle 
\end{align*}
Hence, we find the necessary condition on $(\bfm, \dot{\bfm}) \in T \bfM$: 
\begin{equation} \label{constraint_cavity}
    \sum_k \int_{\pa \Omega_{m_k}} \big\langle V_{m_k} , \dot{m}_k \big\rangle  = 0 
\end{equation}
This condition echoes the one that the fluid velocity $u$ has to satisfy: for all $t$,  
\begin{equation} \label{constraint_cavity_2}
   \sum_k \int_{\pa \Omega_{m_k(t)}} \big\langle V_{m_k(t)} , m'_k(t) \big\rangle  = 0  
   \end{equation}
Hence, we have somehow to restrict the analysis to elements $(\bfm, \dot{\bfm}) \in T \bfM$ satisfying \eqref{constraint_cavity}. Note that for each $\bfm \in \bfM$, the map 
$$ \dot{\bfm} \rightarrow   \sum_k \int_{\pa \Omega_{m_k}} \big\langle V_{m_k} , \dot{m}_k \big\rangle$$
is a linear form on $T_\bfm \bfM$, so that either it is identically zero, or its kernel is a hyperplane of $T_\bfm \bfM$, that we may denote $F_\bfm$. In the latter case, with our notations, $\dim F_\bfm  = p-1$. If we are in the favorable (generic) situation where the second case holds for all $\bfm$, and where $\sqcup_{\bfm \in \bfM} \{\bfm\} \times F_\bfm$  is a smooth sub-bundle of the tangent bundle $T\bfM$, one can adapt the proof of well-posedness above. First, for any $\bfm, \dot{\bfm}$ satisfying \eqref{constraint_cavity}, one can easily modify the proof of Lemma \ref{lemma_phi_m} to show existence of a solution to \eqref{system_phi_m_bis}. Then, given $\big(\bfm_0, \dot{\bfm_0}\big)$ satisfying \eqref{constraint_cavity} (the initial data for the ODE system),   one can introduce as before in the vicinity of $\bfm_0$ a family of smooth vector fields  $\mathbf{E}_1$, \dots, $\mathbf{E}_p$ and a dual family of smooth one-forms $\mathbf{E}_1^*, \dots, \mathbf{E}_p^*$, such that \eqref{basis_dual_basis} holds, but this time with the extra requirement that 
$$ \forall \bfm, \quad \big(\mathbf{E}_1(\bfm), \dots, \mathbf{E}_{p-1}(\bfm)\big) \: \text{ is a basis of $F_\bfm$}. $$
One can further decompose 
$$ u(t,\cdot) =  \sum_{i=1}^{p-1} \dot{\bfm}_i(t)  \, \na \phi^i_{\bfm(t)}, \quad \text{where for all $\bfm$ in $V$, } \phi^i_{\bfm} = \phi_{\bfm,\mathbf{E}_i(\bfm)},$$ 
and proceeding exactly as in Paragraph \ref{sec_ODEs}, we derive a system of ODE's of the form: 
$$ \forall i=1\dots p-1, \quad \dot{\bfm}_i'(t) = \tilde{G}_i(\bfm(t), \bfm'(t))$$
It can be reformulated as a usual second order ODE system of the form $\bfm''(t) = \mathcal{G}(\bfm(t), \bfm'(t))$: indeed, one can write artificially $\bfm'(t) =  \sum_{i=1}^p \dot{\bfm}_i(t) \mathbf{E}_i(\bfm)$ and add the trivial equation 
$$\dot{\bfm}'_p(t) = 0.$$
It has therefore  a unique local in time solution starting from  $\big(\bfm_0, \dot{\bfm_0}\big)$. 

\medskip
As an example, let us consider the case of two spherical bubbles in a bounded cavity: 
$$ B_1(t) = B(c_1(t),r_1(t)) \quad \text{ and } \:  B_2(t) = B(c_2(t),r_2(t))  $$
\begin{align*}
& M = \big\{ m = (c,r), \quad c \in \R^3, \quad r > 0, \quad \overline{B(c,r)} \subset \Omega  \big\}, \\
 \text{ and } \: & \bfM = \big\{ (m_1, m_2) \in M^2, \: |c_1 - c_2| > r_1 + r_2 \big\}. 
 \end{align*}
From the computations of Paragraph  \ref{sec_spherical_bubbles}, we have that 
$$ \big\langle V_m(x), (\dot{c}, \dot{r}) \big\rangle = \dot{c} \cdot n(x) + \dot{r}$$
From there, using that $\int_{\pa \Omega_m} \dot{c} \cdot n = 0$, $\int_{\pa \Omega_m} \dot{r} = 4\pi r^2 \dot{r}$, we deduce that \eqref{constraint_cavity} reads 
\begin{equation*} 
    r_1^2 \dot{r}_1 + r_2^2 \dot{r}_2 = 0. 
\end{equation*}
while \eqref{constraint_cavity_2} yields 
\begin{equation} \label{relation_dotr1_dotr2} 
r_1(t)^2 r'_1(t) + r_2(t)^2 r'_2(t) = 0. 
\end{equation}
which implies that 
\begin{equation} \label{relation_r1_r2}
r_2(t)^3 = - r_1(t)^3 + r_{1,0}^3 + r_{2,0}^3 
\end{equation}
(where $r_{1,0}$ and $r_{2,0}$ are the initial radii of the two bubbles). 
Then, we can introduce for $k=1,2$ and $i=1\dots 3$ the solutions $\phi^{k,i}_\bfm$ of 
\begin{align*} 
& \Delta_x \phi^{k,i}_\bfm(t, \cdot) = 0 \quad \text{ in } \: \Omega\setminus \cup_{k'=1}^2 \Omega_{m_{k'}}, \\
& \pa_n \phi^{k,i}_\bfm\vert_{\pa \Omega_{k'}} = \delta_{k,k'} e_i \cdot n, \quad \pa_n \phi^{k,i}_\bfm\vert_{\pa \Omega} = 0, 
\end{align*}
and we complete it by the solution $\phi^r_\bfm$ of 
\begin{align*} 
& \Delta_x \phi^r_\bfm(t, \cdot) = 0 \quad \text{ in } \: \Omega\setminus \cup_{k'=1}^2 \Omega_{m_{k'}}, \\
& \pa_n \phi^r_\bfm\vert_{\pa \Omega_1} = 1, \quad \pa_n \phi^r_\bfm\vert_{\pa \Omega_2} = - \frac{r_1^2}{r_2^2}, \quad  \pa_n \phi^r_\bfm\vert_{\pa \Omega} = 0. 
\end{align*}
so that 
$$ u = \na \left(  \sum_{k=1}^2\sum_{i=1}^3   c'_{k,i}(t)  \phi^{k,i}_{\bfm(t)} + r'_1(t) \phi^r_{\bfm(t)} \right) $$
Eventually, one gets a second order ODE system of the form 
\begin{align*}
   & c_1'' = F_1(c_1,c_2,r_1,r_2,c_1', c_2', r_1',r_2'), \quad  c_2'' = F_2(c_1,c_2,r_1,r_2,c_1', c_2', r_1',r_2'), \\
   & r_1'' = F_r(c_1,c_2,r_1,r_2,c_1', c_2', r_1',r_2')
\end{align*}
which after replacing $r'_2$ and $r_2$ through \eqref{relation_dotr1_dotr2} and \eqref{relation_r1_r2} yields a closed system of ODEs on $c_1,c_2,r_1$. 

\section*{Acknowledgements}
Both authors acknowledge the support of Project ComplexFlows (ANR-23-EXMA-0004), and of the Project Complexcité of
Université Paris Cité. C.B. acknowledges the partial support by the Agence Nationale pour la Recherche grant CRISIS (ANR-20-CE40-0020-01). D.G.-V. acnowledges the support of the ANR Project Bourgeons, grant ANR- 23-CE40-0014-01, and of the ANR-DFG Project Suspensions, grant ANR-24-CE92-0028.

{\small
\bibliographystyle{apalike}
\bibliography{weak_solutions}
}
\end{document}